\documentclass[a4paper]{amsart}

\usepackage{amsfonts}
\usepackage{amssymb}
\usepackage{graphicx}

\usepackage{cite}

\theoremstyle{plain}
\newtheorem{theorem}{Theorem}
\newtheorem*{theorem*}{Theorem}
\newtheorem{lemma}{Lemma}[section]
\newtheorem{proposition}[lemma]{Proposition}

\theoremstyle{definition}
\newtheorem{definition}[lemma]{Definition}
\newtheorem{algorithm}[lemma]{Algorithm}

\theoremstyle{remark}
\newtheorem{remark}[lemma]{Remark}
\newtheorem*{remark*}{Remark}

\newcommand{\kw}[1]{\textbf{#1}}

\newcommand{\bR}{{\mathbb R}}
\newcommand{\bN}{{\mathbb N}}
\newcommand{\bZ}{{\mathbb Z}}

\newcommand{\cA}{{\mathcal A}}
\newcommand{\cB}{{\mathcal B}}

\newcommand{\cE}{{\mathcal E}}
\newcommand{\cF}{{\mathcal F}}
\newcommand{\cG}{{\mathcal G}}

\newcommand{\cP}{{\mathcal P}}
\newcommand{\cR}{{\mathcal R}}

\newcommand{\cT}{{\mathcal T}}
\newcommand{\cU}{{\mathcal U}}
\newcommand{\cV}{{\mathcal V}}
\newcommand{\cW}{{\mathcal W}}

\DeclareMathOperator{\interior}{int}
\DeclareMathOperator{\card}{card}
\DeclareMathOperator{\cl}{cl}
\DeclareMathOperator{\diam}{diam}

\DeclareMathOperator{\GCD}{GCD}

\begin{document}

\title[Finite resolution dynamics]{Finite resolution dynamics}

\author[S. Luzzatto]{Stefano Luzzatto}
\address{Stefano Luzzatto \\
Mathematics Department, Imperial College \\
180 Queen's Gate \\ London SW7 2AZ, United Kingdom}

\curraddr{Abdus Salam, International Centre
for Theoretical Physics, Strada Costiera 11, 34151 Trieste, Italy}
\email{luzzatto [at] ictp.it}
\urladdr{http://www.ictp.it/\symbol{126}luzzatto/}

\author[P. Pilarczyk]{Pawe\l{} Pilarczyk}
\address{Pawe\l{} Pilarczyk \\
Centro de Matem\'atica, Universidade do Minho \\
Campus de Gualtar \\ 4710-057 Braga, Portugal}
\email{pawel.pilarczyk [at] math.uminho.pt}
\urladdr{http://www.pawelpilarczyk.com/}
\thanks{\emph{Acknowledgements:}
We express our gratitude to several colleagues
who provided us with constructive feedback on our work.
We also thank the anonymous referees for their very careful reading
of the paper and for several suggestions which allowed us
to significantly improve the presentation
and even the overall point of view of the results.
This research was partly funded by the Royal Society (UK).
All the computations for this paper were conducted on a computer
funded by the Japan Society for the Promotion of Science (JSPS),
Grant-in-Aid for Scientific Research (No.\ 1806039),
Ministry of Education, Science, Technology, Culture and Sports, Japan.
\\
Communicated by Peter Kloeden}


\begin{abstract}
We develop a new mathematical model for describing a dynamical system
at limited resolution (or finite scale),
and we give precise meaning to the notion of a dynamical system
having some property at all resolutions coarser than a given number.
Open covers are used to approximate the topology
of the phase space in a finite way,
and the dynamical system is represented
by means of a combinatorial multivalued map.
We formulate notions of transitivity and mixing
in the finite resolution setting
in a computable and consistent way.
Moreover, we formulate equivalent conditions
for these properties in terms of graphs,
and provide effective algorithms
for their verification. As an application we show
that the H\'enon attractor is mixing
at all resolutions coarser than~$10^{-5}$.
\end{abstract}

\maketitle


\section{Introduction and statement of results}
\label{sec:intro}

The theory of dynamical systems has experienced tremendous growth and
development throughout the recent decades.
Nevertheless, in spite of the undeniable importance of the results
describing complicated dynamical properties
many of these
results are abstract existence or genericity results.
This can be problematic for example in applications,
where generally more concrete and
quantitative results are desirable,
especially if one is studying the properties
of some very specific objects for which analytic methods fail.
It is therefore natural to try to use
numerical and computational techniques, and indeed there is
an enormous amount of literature in this direction.
Unfortunately, this approach is limited to
\emph{finite resolution} and to \emph{finite time} and is therefore
unsuitable for direct rigorous verification of infinite time or
infinitesimal scale properties without additional theoretical
interpretation of the results of computations.

This combination of numerical methods and theoretical interpretation
can indeed be remarkably successful in certain situations.
In many cases non-trivial mathematics can reduce the problem
to the verification of
certain ``inclusion conditions'' or the existence
of certain geometric structures
which can be rigorously verified numerically;
then this information can be used to develop symbolic dynamics
or prove existence of certain trajectories,
e.g., heteroclinic connections.
This approach includes the application of some topological methods,
like the fixed point index or the Conley index.
We mention for example
\cite{Ara07, AraKalKokMisOkaPil09, EckKocWit84, Gal02, Hru06,
KapZgl03, KapSim07, Lan82, MisMro95, Mro99, Pil99, Pil03,
PilSto08, Szy97, Tuc99, Zgl04} and refer the reader
to references therein for further information.
However, the bottom line is that by the intrinsic limitations
of numerical computation, only ``robust'' phenomena,
i.e., phenomena which persist for small perturbations of the system,
can be proved in this way. Indeed, the requirement that all estimates
be rigorously bounded, necessarily implies that there is always
a little margin of perturbation in which they continue to hold true;
see \cite{Mro96} for an interesting discussion of the properties
which can be rigorously proved numerically,
referred to as ``inheritable properties''.
This is of course in many ways also one of the strengths
of the methods rather than just a limitation,
but it does mean that ``unstable'' phenomena,
i.e., dynamical features which exist for one system
but perhaps not for nearby systems, are essentially \emph{undecidable},
both in a theoretical as well as in a practical sense: see \cite{ArbMat04}
for a formal discussion of this problem. The best we can hope for in
this situation is to be able to prove that such phenomena occur
for ``some'', or perhaps even ``many'',
systems close to the one of interest,
but even this strategy is sometimes extremely hard to implement.

One specific example of this situation is the
famous and very well studied H\'enon family
of two-dimensional diffeomorphisms given by
\[
H_{a,b} (x,y) = (1 + y - a x^{2}, b x).
\]
This family was introduced by H\'enon in \cite{Hen76}
where he focused particularly on the parameter values
$a = 1.4$, $b = 0.3$, now commonly referred to as the
``classical'' parameter values.
\begin{figure}[htbp]
\includegraphics[height=3.5cm]{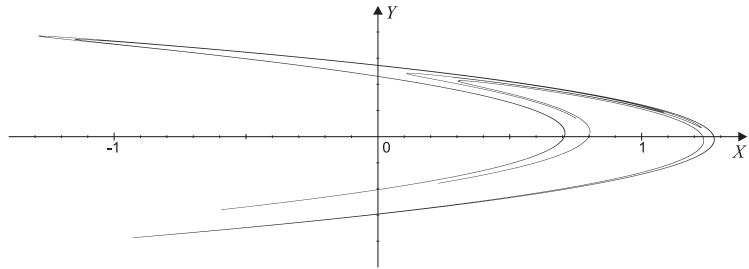}
\label{fig:henon}
\caption{A numerical approximation
of the attractor of the H\'enon map for the classical parameter values.}
\end{figure}
For these parameter values, non-rigorous
computer simulations suggest that most orbits converge
to a transitive attractor with a complicated fractal
geometric structure (see Figure~\ref{fig:henon}),
and exhibit chaotic dynamics.
While the existence of an attractor follows
by relatively elementary arguments,
it turns out that it is extremely difficult,
and indeed arguably \emph{impossible}, to make any rigorous
assertions about the infinitesimal structure
of the attractor.
Indeed, the geometry of the H\'enon map suggests the occurrence
of tangencies between stable and unstable manifolds, and in fact
in \cite{AraMis06} it is proved that
for $b$ sufficiently close to \( 0.3 \) there exists
some parameter \( a\in [1.392419807915, 1.392419807931] \)
for which a homoclinic tangency occurs (this is proved using
some clever application of Conley index theory combined
with some numerical ``inclusion'' estimates as discussed above).
Such tangencies are
a well known source of instability and bifurcations
and associated to a variety of dynamical phenomena
such as infinitely many sinks \cite{New74, PalTak93},
``small'' stochastic attractors \cite{MorVia93, Col98}
and others \cite{GonTurShi03}.
In particular the dynamics is extremely unstable and
there seems to be no hope to establish the actual dynamics
of any single given parameter.
For small values of \( |b| \) it is nevertheless possible to
show that
there is
a \emph{positive Lebesgue measure} of parameters \( a \)
for which the dynamics is stochastic in a very well defined sense,
thus showing that stochastic dynamics has in some sense
\emph{positive probability}.
This was first proved for \( b=0 \) in \cite{Jak81},
with some generalization to other one-dimensional maps
in \cite{Tsu93, Thu99, LuzTuc99, LuzVia00, PacRovVia98},
and for \( b\neq 0 \) in \cite{BenCar91, BenYou93},
with some generalizations in \cite{MorVia93, WanYou01}.
A quantitative approach was developed
in \cite{LuzTak06} with some explicit measure bounds.
Unfortunately,
further extensions of these results to larger values of \( |b| \),
to include for example the classical parameter values,
seem for the moment still out of reach.
We also emphasize that, as mentioned above,
these results still only provide a probabilistic description
of the parameter space, and do not in general
yield information about any particular given parameter.

Our goal in this paper is to approach the problem
from a radically different point of view,
which provides a different kind of information
and is hopefully much more flexible and versatile,
and more widely applicable.
The underlying philosophy is that only finite precision
is usually of practical meaning for applications,
and all phenomena which take place below certain resolution
are of no real importance, either because of the limited accuracy
with which a mathematical model describes a physical system
of interest (due to noise or truncations, for example),
or because of the finite precision of measuring devices
which provide observable properties of the system.
Our goal is therefore to sketch the beginnings of a
theory of \emph{finite resolution dynamics}
and to give an example of a non-trivial implementation
of the concepts developed in such a theory.
The starting point is defined by two basic ``axioms''
or ``principles'' which finite resolution dynamics
should verify:
\begin{quote}
\begin{center}
\textbf{(A)\emph{ Computability}}
\quad and \quad \textbf{ (B) \emph{Consistency}}
\end{center}
\end{quote}
These principles can in practice be implemented in a variety
of ways, but they embody the fundamental ideas that:
(A) dynamical properties should be formulated
in such a way that they are rigorously verifiable
using computational methods, in particular
they should be related to finite time
and finite resolution properties of the system;
(B) computations at a finer scale
should yield more valuable results in the sense that
if a dynamical property holds at some finite scale
then it should also hold for any coarser representation
of the same dynamics.

As an example of the application of these ideas
we define a notion of \emph{combinatorial mixing} which,
in the framework of finite resolution dynamics,
is analogous to the standard notions of topological
or measure-theoretical mixing for topological
or measurable systems, respectively.
We shall prove the following

\begin{theorem}
\label{thm:henon}
The H\'enon attractor is mixing at all resolutions \(> 10^{-5} \).
\end{theorem}

We emphasize that, once the appropriate definitions have been put
in place, this is a non-trivial rigorous statement
about the dynamics of the H\'enon map
for the classical parameter values, obtained by computer-assisted methods.

The structure of the paper is as follows.
In Section~\ref{sec:defs} we introduce the definitions
which are at the foundations of our theory
of finite resolution dynamics, and we state and prove
our main abstract result. More specifically,
we give a precise formulation of Axioms (A) and (B),
define what we mean by the statement
that a certain abstract property is verified
at all resolutions coarser than some \( \varepsilon \),
and demonstrate the key fact that this can be rigorously proved
by computer-assisted arguments.
In Section~\ref{sec:transmix} we
 formulate natural notions of transitivity and mixing
in the finite resolution setting and prove that
these finite-resolution definitions satisfy
the required consistency conditions (B).
In Section~\ref{sec:algor} we provide explicit algorithms
for the verification of transitivity and mixing
at finite resolution, and thus in particular provide
a constructive proof that they satisfy
the computability conditions (A).
In Section~\ref{sec:appl} we describe
the particular application of our theory
to the H\'enon map and we prove that this map
is mixing at some finite but very fine scale.

This paper is accompanied by efficient and flexible software
programmed in the C++ language and made freely available
at~\cite{WWW} together with raw data for the applications
discussed in Section~\ref{sec:appl}.


\section{Abstract theory}
\label{sec:defs}

One possible approach to represent a dynamical system
in a finite way is based on \emph{partitioning},
that is, subdividing the phase space into a finite family
of compact sets with nonempty and disjoint interiors.
This approach goes back several decades and has proved
extremely successful for describing certain classes
of systems such as uniformly hyperbolic diffeomorphisms
and flows \cite{Bow70, BowRue75, Sin68}
and has been recently further developed
and used in various applications
(see \cite{AraKalKokMisOkaPil09, MisMro95,
MroPil02, Pil99, Szy97} for some examples).
However, working with this kind of discretization
in more general systems
may not well reflect the actual topology of the phase space.
Therefore, in this paper we introduce the idea
of using a structure more closely related
to the topology of the phase space: an \emph{open cover}.
As we shall see below, the non-trivial overlap
between elements of an open cover provides a crucial ingredient
for the development of a formal theory of finite resolution dynamics.

\subsection{Open covers}
\label{sec:covers}

Here and for the rest of the paper, \( X \) will always denote
a metric space; without loss of generality,
we shall assume that the metric is bounded.
This generality is convenient, because in applications,
like those considered in Section~\ref{sec:appl},
it may be necessary to work with the space $X$
being a bounded, not necessarily open, subset of $\bR^n$
with the induced metric.
An open ball centered at a point $x \in X$ and of radius $r > 0$
will be denoted by $B (x, r)$.
For a subset \( U \subset X \), its \emph{diameter} $\diam U$
is the supremum of the distances between any two of its elements.
Since $X$ is bounded, the diameter of any subset of $X$ is finite.
For a family $\cA$ of subsets of $X$
we denote their union $\bigcup_{A\in \cA} A$ by $|\cA|$.

\begin{definition}
A finite family $\cU$ of open subsets of $X$
such that $X = |\cU|$
is called a \emph{cover} of $X$.
\end{definition}

We are going to use the elements of $\cU$
as a finite approximation of the topology on~$X$.
In general, the cover \( \cU \) provides a better approximation if
it consists of smaller elements.

\begin{definition}
\label{def:outres}
The \emph{outer resolution}
\[
\cR^+ (\cU) := \max \{\diam (U) : U \in \cU\}
\]
of a cover $\cU$
is the maximal diameter of its elements.
\end{definition}

Intuitively, if two points $x, y \in X$ are at a distance
larger than the outer resolution of a cover
then they are well distinguished from each other
by the cover since they must belong to
distinct elements of the cover.

\begin{remark}
\label{rem:totbd}
For a totally bounded metric space (e.g., a bounded subset
of the Euclidean space with the inherited metric),
there exist covers with arbitrarily small outer resolution.
\end{remark}

\begin{definition}
\label{def:fincov}
We say that a cover $\cU_1$ of $X$
is \emph{finer} than another cover $\cU_2$ of $X$
(or, equivalently, $\cU_2$ is \emph{coarser} than $\cU_1$)
if every element of $\cU_1$ is contained in some element of $\cU_2$,
and every element of $\cU_2$ contains some element of $\cU_1$.
\end{definition}
Note that in this definition we require a mutual relation
between the elements of $\cU_1$ and $\cU_2$ in order to ensure
that $\cU_2$ is indeed a coarser description of the topology of $X$
and does not recognize any finer topological structures of the space
than $\cU_1$ does.
This relation between (essential) covers defines a partial order
on the space of (essential) covers of \( X \), so we can write
\( \cU_{1} \prec \cU_{2} \)
if \( \cU_{1} \) is finer than \( \cU_{2} \).

\medskip
In this paper we shall focus on \emph{essential} covers,
which are in some sense minimal,
that is, no element can be removed from them without leaving
a considerable part of $X$ outside the cover.
\begin{definition}
\label{def:esscov}
A cover $\cU$ of $X$ is \emph{essential}
if there exists $\varepsilon > 0$ such that
every $U \in \cU$ contains some point $x \in X$
such that $B (x, \varepsilon) \subset U$ and
$B (x, \varepsilon) \cap W = \emptyset$
for all $W \in \cU \setminus \{U\}$.
\end{definition}
Notice that if a cover is not essential,
i.e., there is an element \( U \) of the cover
that does not satisfy the condition above
for any \( \varepsilon > 0 \), then $ U$ must be
contained in the union of closures of the other elements
of the cover, and thus is in some sense superfluous;
therefore, covers that are not essential
may introduce a false feeling of the topology
of the space, and indeed, as will be made clear
in the sequel, some features of our theory
do not go through without this assumption.
Because of this reason, \textbf{from now on we shall
always work with essential covers}.

\begin{remark}
\label{rem:esscov}
From any cover it is possible to create an essential one
(without increasing $\cR^+$, but possibly decreasing the number
of elements); see Appendix~\ref{app:esscov} for details.
\end{remark}

\begin{remark}
For essential covers, the definition of finer and coarser covers given in
Definition \ref{def:fincov} reduces to the following:
A cover $\cU_1$ of $X$ is \emph{finer} than another cover $\cU_2$ of $X$
(or, equivalently, $\cU_2$ is \emph{coarser} than $\cU_1$)
if every element of $\cU_1$ is contained in some element of $\cU_2$.
Indeed, this automatically implies the second part:
every element of \( \cU_{2} \) contains some elements of \( \cU_{1} \).
To see this, just suppose by contradiction that there exists
some \( U_{2}\in \cU_{2} \) which does not fully contain
any element of \( \cU_{1} \).
Then, since every element of \( \cU_{1} \) is contained
in some element of \( \cU_{2} \),
it follows that \( \cU_{2}\setminus \{U_{2}\} \) is still a cover,
contradicting the fact that it is essential.
We thank Francesca Aicardi for this observation.
\end{remark}

\subsection{Combinatorial maps}
\label{sec:maps}

We shall use the symbol \( \multimap \) to denote
a possibly multivalued map between two sets.
Let \( X \) be a bounded metric space
and \( \cU \) an open cover of \( X \).

\begin{definition}
A \emph{combinatorial map} is a multivalued map
$\cF \colon \cU \multimap \cU$.
\end{definition}

Since \( \cU \) is finite, $\cF$ is a finite object
that can be represented in a purely combinatorial way.
If \( f\colon X \to X \) is a map, we can use combinatorial maps
to approximate the map \( f \).

\begin{definition}
\label{def:repr}
We say that a combinatorial map $\cF \colon \cU \multimap \cU$
is a \emph{representation} of a map $f \colon X \to X$
if for every $U \in \cU$ we have
\[
\cF (U)\supseteq \{W \in \cU : W \cap f (U) \neq \emptyset\}.
\]
\end{definition}

The ideal situation would be to have a representation
where we have an equality
\(\cF (U) = \{W \in \cU : W \cap f (U) \neq \emptyset\}\)
but computing such a representation
is generally not possible in practice
due to numerical approximation and computer round-off
errors in the calculation of a guaranteed outer bound for \( f(U) \):
see Remark~\ref{rem:minrep} below.
At the other extreme we have the trivial combinatorial map
which maps each \( U \) to all elements of the cover.
This is a combinatorial representation of \( f \)
but gives no information whatsoever about \( f \).
To control in some way this issue, and keeping in mind also that
generally \( \cF \) gives a better approximation
if it is defined on a finer cover, we introduce the following

\begin{definition}
\label{def:outresF}
The \emph{outer resolution}
\[
\cR^+ (\cF) := \max \{\diam U, \diam |\cF (U)| : U \in \cU\}
\]
of a combinatorial map
$\cF \colon \cU \multimap \cU$ is the larger of the maximum diameter
of the elements of the cover \( \cU\) and of the maximum
diameter of their images.
\end{definition}

If $f$ is a continuous map
with some Lipschitz constant $L \geq 0$
and the elements of $\cU$ are of similar size
(e.g., balls of the same radius),
then in principle one should expect
$\cR^+ (\cF) \approx \max\{1, L\} \cR^+ (\cU)$
for a reasonable representation $\cF$ of $f$.

\begin{definition}
\label{def:finmap}
Let $\cU_1$ and $\cU_2$ be two covers of $X$.
We say that $\cF_1\colon \cU_1 \multimap \cU_1$
is \emph{finer} than $\cF_2\colon \cU_2 \multimap \cU_2$
 (or, equivalently, \( \cF_{2} \) is \emph{coarser} than \( \cF_{1}\))
if $\cU_1$ is finer than $\cU_2$
and for every $U_1 \in \cU_1$
and every $U_2 \in \cU_2$ such that $U_1 \subset U_2$,
every element of $\cF_1 (U_1)$
is contained in some element of $\cF_2 (U_2)$.
\end{definition}

Intuitively, $\cF_1$ is finer than $\cF_2$
if the images of $\cF_1$ are smaller
than those of $\cF_2$. Indeed, it follows immediately
from the definition that if $\cF_1$ is finer than $\cF_2$
then $|\cF_1 (U_1)| \subset |\cF_2 (U_2)|$
for all $U_1 \in \cU_1$ and $U_2 \in \cU_2$
such that $U_1 \subset U_2$.
A key point of our approach is to develop some dynamical notions
which are actually computable. In particular,
to take advantage of the notions of combinatorial maps
we need to be able to compute such maps.

\begin{definition}
$f$ is \emph{\(\varepsilon\)-computable}
if there exists a cover $\cU$ of $X$
and a method for computing
a combinatorial representation $\cF \colon \cU \multimap \cU$
of $f$ with \( \cR^+(\cF) \leq \varepsilon\).
\end{definition}

This essentially depends on how explicitly we know
the map \( f \) and is generally not a serious issue:
see Remark~\ref{rem:comp}. We conclude this section
with a series of remarks concerning the definitions above.

\begin{remark}
Unlike in the case of covers, the relation of being finer
does not define a partial order between combinatorial maps.
Consider the following example:
$X = \{a, b, c\}$ with the discrete topology,
$\cU_1 = \big\{\{a\}, \{b\}, \{c\}\big\}$,
$\cU_2 = \big\{\{a, b\}, \{c\}\big\}$,
$\cU_3 = \big\{\{a, b\}, \{b, c\}\big\}$,
$\cF_1 \colon \{a\} \mapsto \big\{\{a\}\big\},
\{b\} \mapsto \big\{\{a\}\big\}, \{c\} \mapsto \big\{\{c\}\big\}$,
$\cF_2 \colon \{a, b\} \mapsto \big\{\{a, b\}\big\},
\{c\} \mapsto \big\{\{c\}\big\}$,
$\cF_3 \colon \{a, b\} \mapsto \big\{\{a, b\}\big\},
\{b, c\} \mapsto \big\{\{b, c\}\big\}$.
Obviously, $\cU_1 \prec \cU_2 \prec \cU_3$.
Moreover, $\cF_1$ is finer than $\cF_2$ and $\cF_2$ is finer than $\cF_3$.
However, $\cF_1$ is not finer than $\cF_3$,
because for $U_1 := \{b\} \subset U_2 := \{b, c\}$
the element $\{a\} \in \cF_1 (U_1)$
is not contained in any element of
$\cF_2 (U_2) = \big\{\{b, c\}\big\}$.
\end{remark}

\begin{remark}
\label{rem:minrep}
Every map $f \colon X \to X$ has a unique
\emph{minimal representation} relative to a given cover
\( \cU \) defined by
\[
\cF_{f, \cU} (U) :=
\{W \in \cU : W \cap f (U) \neq \emptyset\}.
\]
This is the obvious ``abstract'' definition
of a representation of \( f \) which would be natural
if it did not have to be explicitly and rigorously computed.
It is minimal in the sense that the image of any $U \in \cU$
by any other representation $\cF$ of $f$ contains
$\cF_{f, \cU} (U)$. We emphasize, however,
that although \emph{some} representation $\cF$ of $f$
can usually be computed in a relatively straightforward manner,
the \emph{minimal} representation for the same map
need not be computable at all. This is because,
in general, when using rigorous numerical methods,
the best one can typically compute
is some outer approximation $B^f_U$
of the image $f (U)$ of each $U \in \cU$.
Therefore, it is usually possible to compute
some $\cW^f_U$ containing all the elements of $\cU$
which intersect $B^f_U$,
good for constructing a representation of $f$.
However, in general it may be impossible to determine
which of the elements of $\cW^f_U$
actually do intersect $f (U)$ and which do not.
This justifies our definition of a representation of $f$
in which we do not require that it is actually
the \emph{minimal} representation,
and we allow that $\cF (U)$
contains some superfluous elements of $\cU$.
\end{remark}

\begin{remark}
The composition of combinatorial maps can be easily defined.
Formally, let \( 2^{\cU} \) denote the set of all possible
subsets of \( \cU \) and define the map
\( \widehat{\cF} \colon 2^{\cU} \to 2^{\cU} \)
as \( \widehat{\cF} (A) = \bigcup_{U\in A}\cF(U) \)
for any \( A \in 2^{\cU} \).
Then for all \( n\geq 1 \) we can write
\( \mathcal F^{n} = \widehat{\cF}^{n} \circ i \),
where \( i \colon \cU \to 2^{\cU} \) is the ``embedding''
\( i(U) := \{U\} \). In the same way we can define
the composition of two different combinatorial maps
\( \cF \) and \( \cG \) as long as they are both defined
on the same cover.
It is easy to see that if $\cF$ and $\cG$ are representations
of $f, g \colon X \to X$, respectively,
then $\cG \circ \cF$ is a representation of $g \circ f$.
In particular,
if $\cF$ is a representation of $f$
then $\cF^n$ is a representation of $f^n$ for every $n > 0$.
\end{remark}


\begin{remark}
\label{rem:compfin}
Note that if $\cF_1, \cG_1 \colon \cU_1 \multimap \cU_1$
are finer than $\cF_2, \cG_2 \colon \cU_2 \multimap \cU_2$,
respectively, then $\cG_1 \circ \cF_1$ is finer than
$\cG_2 \circ \cF_2$.
In particular, if $\cF_1$ is finer than $\cF_2$
then also the same relation holds for iterations of these maps,
that is, $\cF_1^n$ is finer than $\cF_2^n$.
Indeed, take any $U_1 \in \cU_1$ and any $U_2 \in \cU_2$
such that $U_1 \subset U_2$.
Take any $W_1 \in \cG_1 (\cF_1 (U_1))$.
There exists $V_1 \in \cF_1 (U_1)$
such that $W_1 \in \cG_1 (V_1)$.
Since $\cF_1$ is finer than $\cF_2$,
and $U_1 \subset U_2$, for this $V_1 \in \cF_1 (U_1)$
there exists some $V_2 \in \cF_2 (U_2)$
such that $V_1 \subset V_2$.
Since $\cG_1$ is finer than $\cG_2$,
and $V_1 \subset V_2$, for this $W_1 \in \cG_1 (V_1)$
there exists some $W_2 \in \cG_2 (V_2)$
such that $W_1 \subset W_2$.
As a consequence, $W_1 \subset W_2$,
where $W_2 \in \cG_2 (\cF_2 (U_2))$.
Note that this proof does not go through
if we replace ``and every $U_2 \in \cU_2$''
with ``there exists $U_2 \in \cU_2$''
in Definition~\ref{def:finmap}.
\end{remark}

\begin{remark}
\label{rem:comp}
We make some brief remarks on the computability
of combinatorial maps.
If $f \colon \bR^n \to \bR^n$ is a continuous map
defined by means of an explicit formula
involving only elementary operations
(addition, multiplication, etc.)
and simple arithmetic functions
(like the trigonometric functions, $\sqrt{x}$, or $e^x$),
and an open bounded set $X \subset \bR^n$ can be found
such that $f (X) \subset X$,
then one can compute a finite representation
of $f \colon X \to X$ using the concept
of interval analysis \cite{Moo66}.
For a cover $\cU$ of $X$
that consists of products of open intervals,
a set $B^f_U$ containing $f (U)$
can be computed for each $U \in \cU$,
where $B^f_U$ is also a product of open intervals.
Then we define $\cW^f_U$ to be the union
of those elements of $\cU$ which intersect $B^f_U$
(such a set can be computed easily).
The multivalued map
$\cF \colon \cU \ni U \mapsto \cW^f_U \subset \cU$
is then a combinatorial representation of $f$ on $X$.
Obviously, the smaller the elements of the cover $\cU$ are taken,
the smaller the outer resolution of the constructed representation
should be expected.
This argument proves that a finite representation of dynamics
for a wide class of maps is computable at virtually any resolution.
We shall use the idea described above in Section~\ref{sec:appl}
for the analysis of the H\'enon map.
\end{remark}

\subsection{The axioms of finite resolution dynamics}
\label{sec:axioms}

We are now ready to formalize the axioms
of computability and consistency
introduced in Section~\ref{sec:intro},
based on the notion of combinatorial maps.
Let $\cP := \cP (\cF)$ be a predicate
concerning a combinatorial map \( \cF \).

\begin{definition}
We say that \( \cP \) is a \emph{finite resolution property}
if it satisfies the following two conditions:

\begin{description}
\item[(A) Computability]
\( \cP \) is \emph{computable}, that is,
there exists an algorithm which
for any combinatorial map $\cF$
can establish in finite time
whether $\cF$ satisfies \( \cP \) or not.

\item[(B) Consistency]
\( \cP \) is \emph{consistent}, that is,
 for any pair of combinatorial maps $\cF_1$ and $\cF_2$
such that $\cF_1$ is finer than $\cF_2$,
we have $\cP (\cF_1) \implies \cP (\cF_2)$.
\end{description}
\end{definition}

We remark that the definition is given purely
in terms of combinatorial maps with no reference
to the underlying space or any map whose representation $\cF$ might be.
Thus we have an \emph{a priori} notion of what it means
for a property \( \cP \) to be an ``acceptable''
property for the investigation of finite resolution dynamics.
In general we will apply this definition to combinatorial maps
which arise as representations of some particular
map \( f\colon X \to X \) on a metric space \( X \)
and thus obtain some coherent statements
about the ``finite resolution'' dynamics of the map \( f \).

We make here a simple observation which is crucial
to motivate the theory. Part (B) in the definition above
is not in itself sufficient to define what it means
for the property \( \cP \) to hold
at ``all resolutions \( > \varepsilon \)''.
Indeed, assuming that \( \cP \) holds for some
combinatorial representation \( \cF \)
with respect to some cover \( \cU \) with outer resolution
\( \cR^+(\cF) \leq \varepsilon \), it is easy to construct
another combinatorial representation \( \tilde {\cF} \)
with \( \cR^+(\tilde{\cF}) > \varepsilon \)
but which is not coarser than \( \cF \)
and therefore is not automatically
guaranteed to satisfy \( \cP \).
In other words, choosing a cover with a larger resolution
does not necessarily guarantee that the cover is coarser,
and therefore the consistency condition
does not provide any information.
We need some notion which guarantees
that \( \cP \) holds for \textbf{all} combinatorial maps
at sufficiently coarse scales,
based on the fact that \( \cP \) has been verified
at some scale for a \textbf{single}
particular choice of combinatorial map.
To formulate this notion
we need to introduce an additional definition.

\subsection{Inner Resolution}

The following definition and its application constitute
in some sense the key idea of this paper.

\begin{definition}
\label{def:innres}
The \emph{inner resolution} of a cover $\cU$ is
\[
\cR^- (\cU) := \sup \{d \geq 0 : \forall x \in X \;
\exists U \in \cU : B (x, d) \subset U\}.
\]
The \emph{inner resolution} of a combinatorial map
$\cF \colon \cU \multimap \cU$ is
\[
\cR^- (\cF) := \cR^- (\cU).
\]
\end{definition}

Positive inner resolution is a crucial feature
that distinguishes our approach from one
based on partitions, where the inner resolution is $0$.

\medskip
We have used the convention here that \( B(x, 0)=\{x\}\)
so that the inner resolution
is the supremum of the numbers $d > 0$ such that every ball
$B (x, d) \subset X$ is contained in some $U \in \cU$,
or $0$ if such $d > 0$ does not exist.
In a connected space, the quantity $\cR^- (\cU)$ can be interpreted
as the minimal width of overlapping
between adjacent elements of $\cU$.
Intuitively, if the distance between two points $x, y \in X$
is smaller than the inner resolution of a cover
then these points can be identified
by means of belonging to a common element of the cover.

\begin{remark}
\label{rem:lebesgue}
If $X$ is compact then we can apply Lebesgue's number lemma
to any open cover $\cU$ of $X$ in order to know that there exists
a number $\delta = \delta (\cU) > 0$ such that every subset of $X$
whose diameter does not exceed $\delta$ is contained
in some element of the cover. Then obviously
$\cR^- (\cU) \geq \delta / 2 > 0$.
\end{remark}

\begin{remark}
\label{rem:covres}
In certain cases, e.g., if $X$ is a bounded subset
of a Euclidean space, for any $\varepsilon \geq \delta > 0$
it is possible to construct a cover $\cU$ of $X$
such that $\cR^+ (\cU) \leq 2 \varepsilon$
and $\cR^- (\cU) \geq \delta$.
In particular, this is true if there exists a finite number of points
$\{x_1, \ldots, x_n\} \subset X$ such that
$\cU_0 := \{B (x_i, \varepsilon - \delta) : i = 1, \ldots, n\}$
is a cover of $X$. Indeed, let us consider
$\cU := \{B (x_i, \varepsilon) : i = 1, \ldots, n\}$.
Obviously, $\cU$ is a cover of $X$
and $\cR^+ (\cU) \leq 2 \varepsilon$.
Moreover, every $x \in X$ is within the distance
of $r < \varepsilon - \delta$ from $x_k$
for some $k \in \{1, \ldots, n\}$,
and thus $B (x, \delta) \subset B (x_k, \varepsilon)$
by the triangle inequality, so $\cR^- (\cU) \geq \delta$.
\end{remark}

\begin{figure}[htbp]
\includegraphics[height=5cm]{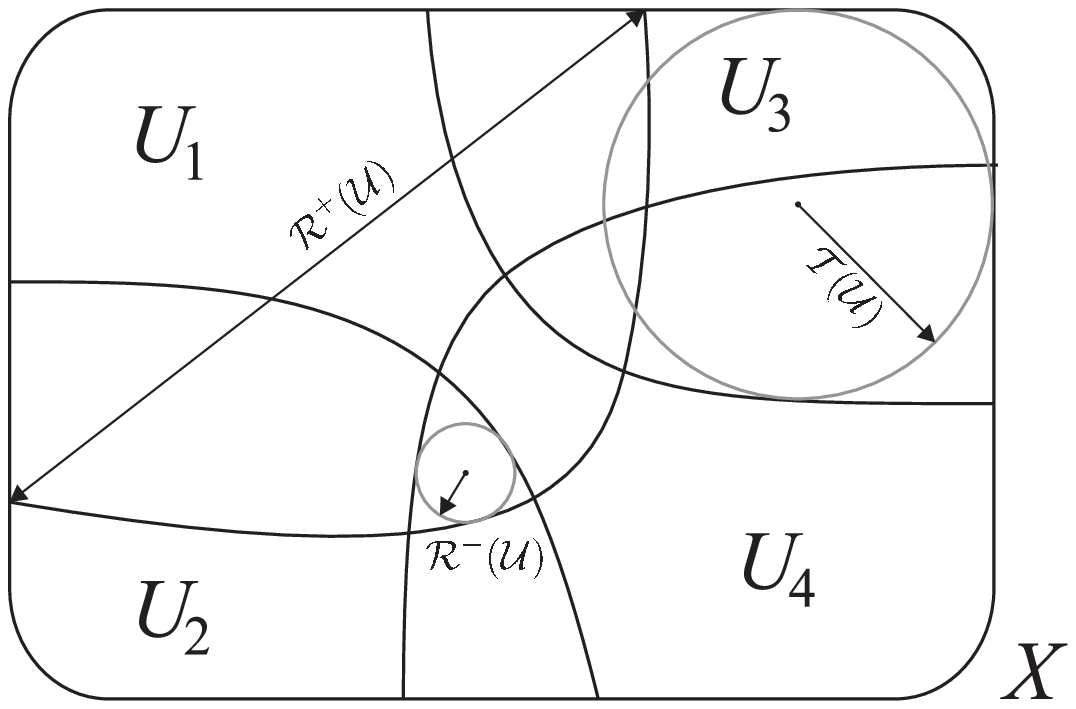}
\caption{An illustration of the quantities \( \cR^+ (\cU) \),
\( \cR^- (\cU) \) in Definitions \ref{def:outres} and~\ref{def:innres},
and \( \cT (\cU) \) in Remark~\ref{rem:resol},
for a sample cover $\cU = \{U_1, U_2, U_3, U_4\}$.}
\label{fig:cover}
\end{figure}

\begin{remark}
\label{rem:resol}
The inner resolution also provides a lower bound
for how ``thick'' the cover elements are
in terms of containing a ball of a big enough radius.
Define \( \cT (\cU) := \sup \{d > 0 : \forall U \in \cU \;
\exists x \in X : B (x, d) \subset U\} \).
(See Figure~\ref{fig:cover} for an elementary example.)
It is easy to see that for an essential cover $\cU$ of $X$
we have \( \cR^- (\cU) \leq \cT (\cU) \).
Indeed, take any positive number $d < \cR^- (\cU)$.
Take any $U \in \cU$. By the assumption that $\cU$ is essential,
there exists $x \in U$ such that for some $\varepsilon > 0$
the ball $B (x, \varepsilon)$ is contained in $U$
and disjoint from all the other elements of the cover $\cU$.
Since $d < \cR^- (\cU)$, for this particular $x$
there exists some $U' \in \cU$
such that $B (x, d) \subset U'$.
However, $U$ is the only element of $\cU$
that contains $x$, so clearly $U' = U$,
and thus $B (x, d) \subset U$.
This shows that for every $U \in \cU$
there exists an $x \in \cU$ such that $B (x, d) \subset U$.
Since $d < \cR^- (\cU)$ can be arbitrarily close to $\cR^- (\cU)$,
the supremum of these numbers
is greater than or equal to $\cR^- (\cU)$.
\end{remark}

\subsection{Finite resolution properties
for all resolutions \( > \varepsilon \)}

We are now ready to state precisely what we mean
when we say that a property holds at all resolutions
coarser than some \( \varepsilon \), and to state
and prove our main abstract theorem.
Let \( \cP \) be a finite resolution property and let
\( f \colon X \to X \) be a map
for which there exists a representation
whose outer resolution is $\leq \varepsilon$.

\begin{definition}
\label{def:allres}
\( f \) \emph{satisfies \( \cP \)
at all resolutions \( > \varepsilon \)}
if \( \cP \) is satisfied for every representation
\( \cF \colon \cU \multimap \cU \) of $f$
with \( \cR^-(\cF)> \varepsilon\).
\end{definition}

We are using here the notion of inner resolution
to quantify the resolution of a cover,
but recall from Remark \ref{rem:covres}
that in many cases the inner and outer resolution
can be chosen arbitrarily close (up to the factor of $2$).
Notice also that the definition requires \( \cP \) to hold
for \emph{any} cover \( \cU \)
and \emph{any} combinatorial representation
\( \cF\colon \cU \multimap \cU \)
as long as \( \cR^-(\cF)> \varepsilon\).
It is therefore \emph{a priori} an unverifiable condition.
However, we have the following

\begin{theorem}
\label{thm:allres}
Suppose \( \cP \) is a finite resolution property.
If $\cP$ holds for a
representation $\cF_0$ of $f$
with $\cR^+ (\cF_0) \leq \varepsilon$
then $\cP$ holds for all
representations $\cF$ of $f$
with $\cR^- (\cF) > \varepsilon$.
Thus $f$ satisfies $\cP$
at all resolutions $> \varepsilon$.
\end{theorem}

\begin{proof}
We present the proof in the form of two lemmas.
\begin{lemma}
\label{lem:fincov}
Let $\cU_1$ and $\cU_2$ be two covers of~$X$.
If $\cR^+ (\cU_1) < \cR^- (\cU_2)$
then $\cU_1 \prec \cU_2$.
\end{lemma}
\begin{proof}
Take any number $r$ such that
$\cR^+ (\cU_1) < r < \cR^- (\cU_2)$.
Let $U_1 \in \cU_1$.
Consider a ball $B (x, r)$ for some $x \in U_1$.
From the fact that $r > \cR^+ (\cU_1)$
it follows that $U_1 \subset B (x, r)$.
Since $r < \cR^- (\cU_2)$,
there exists $U_2 \in \cU_2$
such that $B (x, r) \subset U_2$,
and thus $U_1 \subset U_2$.
Now consider any $U_2 \in \cU_2$.
Then (by Remark~\ref{rem:resol})
there exists a ball $B (x, r) \subset U_2$.
Since $\cU_1$ is a cover of $X$,
there exists $U_1 \in \cU_1$ such that $x \in U_1$.
Since $\diam U_1 < r$, we have $U_1 \subset B(x, r)$,
and thus $U_1 \subset U_2$.
\end{proof}
\begin{lemma}
\label{lem:finmap}
If $\cU_1$ and $\cU_2$ are two covers of $X$,
$\cF_1 \colon \cU_1 \multimap \cU_1$ is a representation
of some map $f \colon X \to X$,
and $\cR^- (\cU_2) > \cR^+ (\cF_1)$,
then any representation
$\cF_2 \colon \cU_2 \multimap \cU_2$ of $f$
is coarser than $\cF_1$.
\end{lemma}
\begin{proof}
By Lemma \ref{lem:fincov},
we have $\cU_1 \prec \cU_2$.
Take any $U_1 \in \cU_1$ and any $U_2 \in \cU_2$
such that $U_1 \subset U_2$.
Take $x \in U_1$. Take a number $r$ such that
$\cR^+ (\cF_1) < r < \cR^- (\cU_2)$.
Then there exists $W_2 \in \cU_2$
such that $B (f (x), r) \subset \cW_2$.
Since $\cF_2$ is a representation of $f$ and $x \in U_2$,
$\cF_2 (U_2)$ contains all the elements of $\cU_2$
which contain $f (x)$, including $W_2$.
On the other hand,
since $\diam |\cF_1 (U_1)| < r$,
every $W_1 \in \cF_1 (U_1)$ is contained
in $B (f (x), r) \subset W_2$.
This proves that indeed $\cF_2$ is coarser than $\cF_1$.
\end{proof}
The statement in the Theorem now follows immediately
from Lemma~\ref{lem:finmap} and the definitions.
\end{proof}

\begin{remark}
\label{lem:advantage}
Lemmas \ref{lem:fincov} and \ref{lem:finmap}
show a tremendous advantage of using open covers
as opposed to partitions, because there is no easy condition
defined in terms of the features of a partition itself
which would guarantee that this partition and any map on it
are coarser than some other partition and some other map,
respectively, both possibly given \emph{a priori} (e.g., as a result
of some computation, as we do in Section~\ref{sec:appl}).
\end{remark}

\begin{remark}
\label{rem:partitions}
If one prefers to do computations for partitions
instead of using open covers, then $\cF_0$ in the assumptions
of Theorem~\ref{thm:allres} can in fact be a multivalued map
on a partition.
Definition~\ref{def:outresF} of the outer resolution
in this case can be applied directly,
Definition~\ref{def:repr} of a representation can be taken
the same and is equivalent to requesting
that $f (U) \subset \interior |\cF (U)|$,
and the predicate $\cP$ can also be the same
as for a combinatorial map on a cover.
Note, however, that the conclusion of Theorem~\ref{thm:allres}
is not valid for partitions, because their inner resolution is $0$.
\end{remark}


\section{Transitivity and mixing}
\label{sec:transmix}

Let \( \cF\colon \cU \multimap \cU \) be a combinatorial map,
and let $\cF^{-1}$ denote its inverse
defined by the following condition:
$V \in \cF^{-1} (U)$ if and only if $U \in \cF (V)$.
We introduce the notions of transitivity and mixing
for combinatorial maps,
further also called combinatorial transitivity
and combinatorial mixing, respectively.

\begin{definition}
$\cF$ is \emph{transitive} if for every
$U, V \in \cU$ there exists $n > 0$
such that $V \in \cF^{-n} (U)$.
$\cF$ is \emph{mixing} if for every $U, V \in \cU$
there exists $N > 0$ such that $V \in \cF^{-n} (U)$
for all $n > N$.
\end{definition}

We shall prove the following

\begin{proposition}
\label{prop:tranmix}
Combinatorial transitivity and combinatorial mixing
are finite resolution properties.
\end{proposition}

To prove the proposition we need to prove the
computability condition (A) and the consistency condition (B)
in the definition of finite resolution property.
Condition (A) says that the properties
of (combinatorial) transitivity and mixing
are algorithmically verifiable
in finite time. This is almost immediate from the definitions
and Lemma~\ref{lem:mixing-all} below,
since combinatorial transitivity and mixing
only require checking a finite number of conditions
(because $\cU$ is finite);
however, in Section~\ref{sec:algor}
we provide explicit algorithms
which can be used to verify the conditions
in an efficient way in practice.
In this section we prove (B).

Notice first that
mixing implies transitivity
but the converse is not true in general.
Moreover, since $V \in \cF^{-n} (U)$
is equivalent to $U \in \cF^n (V)$,
one can equivalently formulate the combinatorial conditions
for the $n$\nobreakdash-th iterate of~$\cF$
instead of the \mbox{$n$\nobreakdash-th} preimage of~$\cF$,
thanks to the symmetry between $U$ and $V$ in these conditions.
We shall use this easy observation transparently in the sequel.
We shall consider transitivity and mixing separately.

\begin{lemma}
\label{lem:trans-mon}
Combinatorial transitivity is a consistent property.
\end{lemma}

\begin{proof}
Let $\cU_1$ and $\cU_2$ be two covers of $X$,
with $\cU_1$ finer than $\cU_2$.
Let $\cF_1 \colon \cU_1 \multimap \cU_1$
and $\cF_2 \colon \cU_2 \multimap \cU_2$
be combinatorial maps.
Assume that $\cF_1$ is finer than $\cF_2$
and that $\cF_1$ is transitive.
We shall show that under these assumptions
also $\cF_2$ is transitive.
Let $U_2, V_2 \in \cU_2$.
Let $B$ be a ball contained in $U_2$
and disjoint from all the other elements of $\cU_2$;
the existence of such a ball follows
from the fact that $\cU_2$ is essential.
Let $U_1, V_1 \in \cU_1$ be such that
$U_1 \cap B \neq \emptyset$ and $V_1 \subset V_2$.
Obviously, in this case $U_1 \subset U_2$.
By the assumption on transitivity of $\cF_1$,
there exists $n > 0$ such that $U_1 \in \cF_1^n (V_1)$.
Since $|\cF_1^n (V_1)| \subset |\cF_2^n (V_2)|$,
we know that $U_1 \subset |\cF^n_2 (V_2)|$.
In particular, $U_1 \cap B \subset |\cF^n_2 (V_2)|$,
and therefore $U_2 \in \cF^n_2 (V_2)$,
because otherwise no part of $B$ could have been covered
by $\cF^n_2 (V_2)$.
Since the choice of $U_2, V_2 \in \cU_2$ was arbitrary,
this proves that $\cF_2$ is transitive.
\end{proof}

\begin{lemma}
\label{lem:mixing-mon}
Combinatorial mixing is a consistent property.
\end{lemma}

Before proving Lemma~\ref{lem:mixing-mon},
we first prove two simple lemmas regarding an equivalent condition
for mixing of combinatorial maps.

\begin{lemma}
\label{lem:mixing-all}
Let $\cF \colon \cU \multimap \cU$ be a combinatorial map,
and let $k > 0$. If $\cF^k (U) = \cU$ for all $U \in \cU$
then also $\cF^n (U) = \cU$ for all $U \in \cU$ and for all $n > k$.
\end{lemma}

\begin{proof}
This follows by induction from the fact that
\( \cF^{k+1}(U) = \bigcup_{V\in\cF(U)}\cF^{k}(V) \)
and that $\cF(U) \neq \emptyset$ and $\cF^{k}(V) = \cU$.
%
\end{proof}

\begin{lemma}
\label{lem:mixing-equiv}
Let $\cF \colon \cU \multimap \cU$ be a combinatorial map.
Then $\cF$ is mixing if and only if
there exists $k > 0$ such that $\cF^k (U) = \cU$ for all $U \in \cU$.
\end{lemma}

\begin{proof}
Let us first assume that $\cF$ is mixing.
Then for every $U, V \in \cU$ there exists $N_{U, V}$
such that $V \in \cF^n (U)$ for all $n > N_{U, V}$.
Since $\cU$ is finite,
the number $k := \max \{N_{U, V} : U, V \in \cU\}$
is well defined and finite.
Obviously, for this $k$,
the set $\cF^k (U)$ contains every $V \in \cU$,
for all $U \in \cU$.

Now suppose that $\cF^k (U) = \cU$ for all $U \in \cU$
and some $k > 0$.
Lemma~\ref{lem:mixing-all}
implies that $\cF^n (U) = \cU$ for all $U \in \cU$
and all $n \geq k$,
which immediately implies that $\cF$ is mixing.
\end{proof}

\begin{proof}[Proof of Lemma~\ref{lem:mixing-mon}]
Let $\cU_1$ and $\cU_2$ be two covers of $X$,
with $\cU_1$ finer than $\cU_2$.
Let $\cF_1 \colon \cU_1 \multimap \cU_1$
and $\cF_2 \colon \cU_2 \multimap \cU_2$ be combinatorial maps.
Assume that $\cF_1$ is finer than $\cF_2$ and that $\cF_1$ is mixing.
We shall show that under these assumptions also $\cF_2$ is mixing.
Take $k > 0$ such that $\cF_1^k (U) = \cU_1$ for all $U \in \cU_1$,
given by Lemma~\ref{lem:mixing-equiv}.
Take any $W \in \cU_2$ and consider $\cF_2^k (W)$.
Since $\cU_1$ is finer than~$\cU_2$,
there exists some $U \in \cU_1$ such that $U \subset W$.
Since $\cF_1$ is finer than~$\cF_2$,
it follows that $|\cF_1^k (U)| \subset |\cF_2^k (W)|$.
The assumption that $\cU_2$ is essential
implies that $\cF_2^k (W) = \cU_2$.
Since the choice of $W \in \cU_2$ was arbitrary,
it follows from Lemma~\ref{lem:mixing-equiv}
that $\cF_2$ is mixing.
\end{proof}


\section{Graph algorithms}
\label{sec:algor}

In this section we provide explicit algorithms for the verification
of combinatorial transitivity and mixing.
In particular, we give a constructive proof of (A).
The first step is to translate the properties
defined in Section~\ref{sec:transmix}
into the language of graphs associated with combinatorial maps.

We recall that a \emph{finite directed graph},
further called \emph{graph} for short,
is a pair $G = (\cV, \cE)$, where $\cV$ is a finite set
whose elements are called \emph{vertices},
and $\cE \subset \cV \times \cV$
is a set of selected (ordered) pairs of vertices.
The elements of $\cE$ are called \emph{edges}.
Combinatorial maps are very naturally encoded as graphs.

\begin{definition}
We say that $G = (\cV, \cE)$ is the graph \emph{associated} with
a combinatorial map $\cF$ on a cover $\cU$ of $X$ if $\cV = \cU$
and $\cE = \{(U,V) \in \cU \times \cU : V \in \cF (U)\}$.
\end{definition}

In this section we explain how the verification of the properties
of combinatorial transitivity and mixing can be reduced
to the verification of certain properties of the associated graphs.
We then describe in some detail specific algorithms that verify
these properties, and we upper provide bounds for their effectiveness.
Before we show such algorithms,
let us explain how their effectiveness is measured.
A running time estimate of an algorithm
is typically given by means of the order
of the number of required primitive operations
as a function of the size of the input,
here given by $|\cV|$, the size of the vertex set, and $|\cE|$,
the size of the edge set.
We use the notation
$O\big(\Psi(|\cV|,|\cE|)\big)$
to indicate that there exist constants $c, n_0 > 0$ such that
for any graph $G = (\cV, \cE)$ for which $|\cV|, |\cE| \geq n_0$,
the number of operations $\Phi (G)$ of the algorithm
applied to the graph $G$ satisfies the inequality
$\Phi(G) \leq c\Psi(|\cV|,|\cE|)$.
This gives an asymptotic upper bound for the worst case running time.
For a more detailed explanation of this notation and of running time
in general, the reader is referred to \cite[\S 3.1]{CorLeiRivSte01}.
In what follows, we say that an algorithm runs in \emph{linear} time
if its worst case running time is $O \big(|\cV| + |\cE|\big)$.

\begin{proposition}
\label{prop:alg}
There exists an algorithm which verifies in linear time
whether a combinatorial map is transitive or not.
There exists an algorithm which verifies in linear time
whether a combinatorial map is mixing or not.
\end{proposition}

Notice that the size \( |\cV| \) of the vertex set is exactly the
number of elements of the cover \( \cU \). Moreover,
in many situations, such as if the map \( f \) is Lipschitz
and the elements of $\cU$ are of similar size, the expected size
\( \cF(U) \) of the image of each element \( U \),
is uniformly bounded, independently of the resolution of the cover.
Therefore, in practice, the size \( |\cE| \) of the edge set
is bounded by a constant multiple of the size \( |\cV| \)
of the vertex set, which is exactly
the number of elements in the cover.
Thus the algorithms we describe, in many situations,
actually run in linear time in the number of the elements
of the cover \( \cU \).
We consider this linear time property
as an important characteristic of our method
since it implies that as computing memory
and power increase, it will be realistically possible
to obtain significant improvements to the scale
at which mixing is verified in systems of interest,
such as the H\'enon map which we consider in this paper
(see, however, additional remarks on this point
in Section~\ref{sec:cost}).

The remaining part of this section is devoted
to the proof of Proposition \ref{prop:alg}.
In particular we shall obtain completely explicit forms
of the required algorithms.
In Section~\ref{sec:graphs} we show that the problem
can be formulated in terms of the verification
of certain properties of directed graphs,
namely strong connectedness and aperiodicity.
The algorithms for verifying these properties are not new.
An algorithm that computes strongly connected components
of a directed graph in linear time
belongs to the canons of graph algorithms
(see \cite[\S 22.5]{CorLeiRivSte01}),
and we use it below to determine whether
a given graph is strongly connected.
An algorithm for proving aperiodicity is given
in \cite{AperWiki} but without a formal proof
of its correctness, which we shall give here.
For completeness, and for the benefit of readers
who are not familiar with graph algorithms
or algorithms in general, we shall give full descriptions
of these algorithms below. They actually require
some non-trivial constructions which we believe
are of independent, albeit relatively technical, interest.
In particular, we provide a fully self-contained
proof of Proposition \ref{prop:alg}.

In Section \ref{sec:repgraphs} we discuss the representation
of graphs in computer's memory.
In Section \ref{sec:dfs} we describe a basic way
of ``exploring'' graphs algorithmically
and of describing certain structures of graphs.
This approach, called ``depth first search'',
underpins both the method
for computing strongly connected components
described in Section \ref{sec:scc},
and the method for verifying aperiodicity
described in Section \ref{sec:aper}.

\subsection{Strongly connected and aperiodic graphs}
\label{sec:graphs}

We now show how the notions of combinatorial transitivity
and mixing can be reduced to certain properties
of the associated graphs.
A \emph{path} in a graph $G = (\cV, \cE)$
is a sequence $(U_i)_{i = 0}^n$
such that $(U_i, U_{i + 1}) \in \cE$
for all $i = 0, \ldots, n - 1$.

\begin{definition}
We say that a graph $G = (\cV, \cE)$ is \emph{strongly connected}
if for every $U, V \in \cV$ there exists a path $(v_0, \ldots, v_k)$
in~$G$ such that $v_0 = U$ and $v_k = V$.
\end{definition}

The \emph{length} of a path $(U_i)_{i = 0}^n$ is $n$.
A \emph{cycle} in $G$ is a path
in which $U_0 = U_n$.
The greatest common divisor of the lengths of all the cycles
in a graph $G = (\cV, \cE)$ is called the \emph{period} of $G$.

\begin{definition}
We say that a graph $G = (\cV, \cE)$ is \emph{aperiodic}
if its period equals $1$.
\end{definition}

\begin{lemma}
\label{lem:graph}
A combinatorial map is transitive if and only if
the associated graph is strongly connected.
A combinatorial map is mixing if and only if
the associated graph is strongly connected and aperiodic.
\end{lemma}

\begin{proof}
Since the existence of a path of length $n$
from $U$ to $V$ in the graph
associated with a combinatorial map $\cF$
is equivalent to $V \in \cF^n (U)$,
the equivalence in the case of transitivity
follows immediately from the definitions.

The statement on mixing can be derived, with certain effort,
from \cite{JabKul03}, but for the sake of completeness
we provide a detailed proof. Assume that
a combinatorial map $\cF \colon \cU \multimap \cU$ is mixing.
Then Lemma~\ref{lem:mixing-equiv} implies that there exists $k > 0$
such that $\cF^k (U) = \cU$ for all $U \in \cU$.
It is thus obvious that $G$ is strongly connected.
Lemma~\ref{lem:mixing-all} implies that also
$\cF^{k + 1} (U) = \cU$ for all $U \in \cU$.
In particular, $U \in \cF^k (U)$ and $U \in \cF^{k + 1} (U)$.
In terms of the graph $G$, this implies that there exist
cycles of the co-prime lengths $k$ and $k + 1$
in $G$ through $U$, and thus $G$ is aperiodic.

Let us now focus on the opposite implication.
Let $T = (U_{j_0}, \ldots, U_{j_t})$ for some $t > 0$
be a cycle in $G$ that runs through all the vertices of $G$
(its existence follows from the strong connectedness of $G$).
Since the $\GCD$ of the lengths of all the cycles in $G$ is $1$,
there exist cycles $C_1, \ldots, C_r$ and $D_1, \ldots, D_s$
of lengths $p_1, \ldots, p_r$ and $q_1, \ldots, q_s$, respectively,
such that $p_1 + \cdots + p_r - q_1 - \cdots - q_s = 1$.
Consider the cycles $T_{kl}$ in $G$ composed of $T$,
$k$ copies of $C_1, \ldots, C_r$,
and $l$ copies of $D_1, \ldots, D_s$.
The length of each such cycle is $t_{kl} = t + kp + lq$,
where $p := p_1 + \cdots + p_r$ and $q := q_1 + \cdots + q_s$.
Since $p = q + 1$, we have $t_{kl} = t + (k + l) q + k$.
If we fix any vertex $U$ of $G$ at the cycle $T$
and consider $k = 0, \ldots, t - 1$ with $l = t - k$,
then we obtain paths from $U$ to $U$ whose lengths
are $t + tq + k = tp + k$.
We can complement these cycles by following $T$
to paths of length $tp + t = t (p + 1)$
that end at each subsequent vertex of $T$.
Since $T$ runs through all the vertices of $G$,
it follows that $\cF^{t (p + 1)} (U) = \cU$.
Since $U \in \cU$ was chosen arbitrarily,
it follows from Lemma~\ref{lem:mixing-equiv} that $\cF$ is mixing.
\end{proof}

\subsection{Representation of graphs}
\label{sec:repgraphs}

The first step in developing algorithms
for studying properties of a graph is to represent
such a graph in computer's memory.
One standard way to do this is as a collection of
\emph{adjacency lists}.
We say that a vertex $v$ is adjacent
to another vertex $u$ if the edge $(u,v)$ is in the graph.
The adjacency-list representation of a graph $G = (\cV, \cE)$
consists of an array of \( |\cV| \) lists,
one for each vertex in \( \cV \).
For each \( u \in\cV \),
the adjacency list of \( u \)
contains all the vertices \( v \in \cV \)
such that \( (u, v)\in \cE \) in certain order.
Searching through this list
or taking one adjacent vertex after another
can be done in the time proportional to the number
of vertices in the list.

Other representations,
such as an ``adjacency matrix'' representation, can also be used
and may be more or less convenient depending on the structure
of the graph under investigation. As a general rule,
the adjacency-list representation is preferable
for \emph{sparse} graphs, that is graphs for which \( |\cE| \)
is of a lower order than its maximum possible value
of \( |\cV|^{2} \), e.g., if it is proportional to $|\cV|$.
This is usually the case for combinatorial maps
and thus we shall use this representation of graphs here.

\subsection{The Depth First Search algorithm}
\label{sec:dfs}

Depth First Search (DFS) is an algorithm for scanning a graph
$G = (\cV, \cE)$. It starts at an arbitrarily chosen vertex,
and follows edges and paths to visit other vertices in the graph.
If no more vertices can be reached in this way
then an arbitrary unvisited vertex is chosen
and the algorithm continues running from that vertex.
Each vertex $u \in \cV$ is assigned
a unique \emph{discovery time} $t_u \in \bN$
which reflects the order in which the vertices are visited,
or discovered.
The general strategy of the algorithm
while visiting a vertex is to take the first
edge leading to an unvisited vertex and to follow it.
If no more vertices can be discovered from a given vertex,
then the algorithm traces back and checks the other edges
leading from the previous vertex.
The name ``depth first search'' comes from the fact
that the search goes as deep into the graph as possible
using the first edge leading to an unvisited vertex,
and the other edges emanating from each vertex
are only checked afterwards.
This is different from another well known strategy,
called ``breadth first search'',
in which all the edges leading from each vertex are checked
before the search continues from the discovered vertices.
The details of the DFS procedure are summarized in the following

\begin{algorithm}
\label{alg:dfs1}
\begin{tabbing}
\hspace{1.5cm}\=\hspace{1cm}\=\hspace{1cm}\=\hspace{1cm}\=\\
\kw{function} {\tt DFS} \\
\kw{input:}
\> $G = (\cV, \cE)$ --- directed graph \\
\kw{code:}
\> {\tt time} \( := 0 \)\\
\> \kw{for each} \( u \in \cV \) \kw{do} \\
\> \> \kw{if not defined} $t_u$ \kw{then} \\
\> \> \> {\tt DFS-Visit} ($u$); \\[6pt]
\kw{function} {\tt DFS-Visit} ($u$)\\
\> {\tt time} $:=$ $\text{\tt time} + 1$; \\
\> $t_u := $ {\tt time}; \\
\> \kw{for each} $v \in \cV$ such that $(u, v) \in \cE$ \kw{do} \\
\> \> \kw{if not defined} $t_v$ \kw{then} \\
\> \> \> {\tt DFS-Visit} ($v$);
\end{tabbing}
\end{algorithm}

This algorithm runs in linear time $O \big(|\cV| + |\cE|\big)$,
because the function {\tt DFS-Visit} is called exactly once
for each vertex in the graph, and all the edges
emanating from each vertex are also checked exactly once.
Note that in order to achieve the linear time in practice,
it is necessary to represent the graph
in such a way that scanning through the set of all the edges
emanating from a given vertex can be done
in the time proportional to the number of these edges,
like in the case of the adjacency-list representation of a graph.
The reader is referred to \cite[\S22.3]{CorLeiRivSte01}
for more details on the DFS algorithm.

\subsection{Strongly connected components}
\label{sec:scc}

One of many applications of the depth first search algorithm
is a method for finding the strongly connected components
of a graph.

The most standard approach (see~\cite[\S22.5]{CorLeiRivSte01})
is to run the DFS algorithm on the graph $G$,
and then to compute depth first trees of $G^T$
(the transpose of $G$, obtained from $G$ by inverting
the direction of all the edges), which turn out
to form the strongly connected components of $G$.
Although this algorithm runs in linear time and memory,
it is inconvenient if applied to huge graphs
because of the need to compute $G^T$,
which takes up as much memory as $G$ itself,
and thus effectively doubles memory usage.

Because of this reason, we use Tarjan's algorithm instead
for computing strongly connected components
of a graph (see~\cite{Tar72, TarWiki}). This algorithm is based
directly upon DFS. In addition to computing the discovery time $t_u$
for each visited vertex, it also computes the lowest discovery time
$l_u$ of all the vertices reachable from that vertex.
This number is computed during the DFS itself,
so that it is always known when needed.
Moreover, during the DFS run, the visited vertices are put
on a stack $S$, and each strongly connected component $C_k$ is taken
from the stack whenever a visited vertex is determined
to ``close'' such a component. For the sake of completeness,
we provide a pseudocode of this algorithm below,
but we refer to~\cite{Tar72} for the proof of its correctness
and for more explanations.

\begin{algorithm}
\label{alg:scc}
\begin{tabbing}
\hspace{1.5cm}\=\hspace{1cm}\=\hspace{1cm}\=\hspace{1cm}\=\\
\kw{function} {\tt SCC} \\
\kw{input:}
\> $G = (\cV, \cE)$ --- directed graph; \\
\kw{code:}
\> {\tt time} := 0; $k := 0$; $S$ := empty stack; \\
\> \kw{for each} $u \in \cV$ \kw{do} \\
\> \> \kw{if not defined} $t_u$ \kw{then} \\
\> \> \> {\tt tarjan} ($u$); \\
\> \kw{return} $(C_1, \ldots, C_k)$\\[6pt]
\kw{function} {\tt tarjan} ($u$) \\
\> {\tt time} := {\tt time} + $1$; $t_u$ := $l_u$ := {\tt time}; \\
\> put $u$ on the top of the stack $S$; \\
\> \kw{for each} $v \in \cV$ \kw{such that} $(u,v) \in \cE$ \kw{do} \\
\> \> \kw{if not defined} $t_v$ \kw{then} \\
\> \> \> {\tt tarjan} ($v$); \\
\> \> \> $l_u := \min (l_u, l_v)$; \\
\> \> \kw{else if} $v$ is in $S$ \kw{then} \\
\> \> \> $l_u := \min (l_u, t_v)$; \\
\> \kw{if} $t_u = l_u$ \kw{then} \\
\> \> $k := k + 1$; \\
\> \> remove the elements from the top of $S$ until $u$ has been \\
\> \> \> removed, too, and put all of the removed elements into $C_k$;
\end{tabbing}
\end{algorithm}

With this algorithm, we now determine whether a given graph
is strongly connected in the most obvious way possible.
Namely, we compute the strongly connected components
and check if there is only one such component,
as made precise below.

\begin{algorithm}
\label{alg:conn}
\begin{tabbing}
\hspace{1.5cm}\=\hspace{1cm}\=\hspace{1cm}\=\hspace{1cm}\=\\
\kw{function} {\tt StronglyConnected} \\
\kw{input:}
\> $G = (\cV, \cE)$ --- directed graph; \\
\kw{code:}
\> $(C_1, \ldots, C_k)$ := {\tt SCC} ($G$); \\
\> \kw{if} $k = 1$ \kw{and} $\card C_1 = \card \cV$ \kw{then} \\
\> \> \kw{return} {\tt true}; \\
\> \kw{else} \\
\> \> \kw{return} {\tt false}.
\end{tabbing}
\end{algorithm}

The following features of the above algorithm are obvious:

\begin{lemma}
\label{lem:conn}
Algorithm \ref{alg:conn} applied to a directed graph $G = (\cV, \cE)$
returns {\tt true} if and only if $G$ is strongly connected.
This algorithm runs in linear time.
\end{lemma}

\subsection{Aperiodic graphs}
\label{sec:aper}

Efficient computation of the greatest common denominator of cycles
in a strongly connected directed graph is less standard,
and the solution suggested in \cite{JabKul03}
yields cubic time. However, there exists an algorithm
which runs in linear time, see~\cite{AperWiki}.
For the sake of completeness,
we describe this algorithm below
and prove its correctness
(note that the latter is not provided in~\cite{AperWiki}).
This algorithm is based upon the standard
DFS (Depth First Search) algorithm
that scans the entire graph starting at an arbitrarily
chosen vertex (recall the assumption that the graph
is strongly connected),
and computes the greatest common divisor
of certain numbers related to the edges of the graph.

\begin{algorithm}
\label{alg:period}
\begin{tabbing}
\hspace{1.5cm}\=\hspace{1cm}\=\hspace{1cm}\=\hspace{1cm}\=\\
\kw{function} {\tt GraphPeriod} \\
\kw{input:}
\> $G = (\cV, \cE)$ --- a strongly connected directed graph; \\
\> $u$ --- an element of $\cV$; \\
\> $d_u \in \bZ$ --- the depth of $u$ in the DFS tree
being constructed; \\
\> $p \in \bZ$ --- the GCD of cycle periods found so far; \\[2pt]
\kw{code:}
\> \kw{for each} $v \in \cV$ \kw{such that} $(u, v) \in \cE$ \kw{do} \\
\> \> \kw{if} \kw{defined} $d_v$ \kw{then} \\
\> \> \> $p$ := GCD ($p$, $d_v - d_u - 1$); \\
\> \> \kw{else} \\
\> \> \> {\tt GraphPeriod}
($G$, $v$, $d_v := d_u + 1$, $p$); \\
\> \kw{return} $p$.
\end{tabbing}
\end{algorithm}

\begin{lemma}
\label{lem:period}
Algorithm \ref{alg:period} applied to a strongly connected
directed graph $G = (\cV, \cE)$, any element $r \in \cV$,
and the numbers $d_r := 0$ and $p := 0$,
returns the greatest common divisor
of the lengths of all the cycles in $G$.
This algorithm runs in linear time.
\end{lemma}

\begin{proof}
Given a strongly connected graph $G = (\cV, \cE)$,
let $T = (\cV, \cE')$ denote the tree obtained by running
Algorithm~\ref{alg:period} on~$G$,
starting with the vertex $r \in \cV$,
which becomes the root of $T$.
The set $\cE'$ consists of all the edges which incur
recursive calls of the function {\tt GraphPeriod}.
The computed depth of each vertex $u \in \cV$
in $T$ is denoted by $d_u$, with $d_r = 0$.
For each edge $e = (u, v) \in \cE \setminus \cE'$,
define $d_e := d_v - d_u - 1$.
Let $c$ be the GCD (greatest common divisor)
of all these numbers $d_e$.
This is the quantity returned by this algorithm.
We shall prove that $c$ equals the GCD
of the lengths of all the cycles in $G$,
further denoted by $c'$.

Let us first prove that $c' | c$.
Consider $e = (u, v) \in \cE \setminus \cE'$.
Since $G$ is strongly connected,
there exists a path $p_{v, r}$ in $G$ from $v$ to $r$;
denote its length by $d$.
Note that the path $p_{r, u}$ in $T$ from $r$ to $u$
is of length $d_u$, and the path $p_{r, v}$ in $T$
from $r$ to $v$ is of length $d_v$.
Since $p_{r, u}$ combined with $e$ and $p_{v, r}$ is a cycle in $G$
of length $d_u + 1 + d$, and $p_{r, v}$ combined with $p_{v, r}$
is a cycle in $G$ of length $d_v + d$, the common divisor $c'$
of the lengths of all the cycles in $G$ must divide
the difference between these lengths, which is $d_v - d_u - 1$.
Therefore, $c$ is a GCD of numbers divisible by $c'$,
and thus $c' | c$.

Let us now prove that $c | c'$.
Let $(v_0, \ldots, v_n)$ be any cycle in $G$.
We shall prove that $c | n$.
Consider the depths of each $v_i$ in $T$.
The difference in the depth traversed by each edge
$e_i = (v_i, v_{i + 1})$ is $\delta_i := d_{v_{i + 1}} - d_{v_i}$,
which equals $1$ if $e_i \in \cE'$.
Since $v_0$ = $v_n$, obviously $d_{v_0} = d_{v_n}$,
and thus $\sum_{i = 0}^{n - 1} \delta_i = 0$.
In particular, $-n = \sum_{i = 0}^{n - 1} (\delta_i - 1) =
\sum_{i = 0}^{n - 1} (d_{v_{i + 1}} - d_{v_i} - 1)$.
If $e_i \in \cE'$ then the corresponding item in the sum is zero,
otherwise it is divisible by $c$.
Therefore, $c | n$.

The observation that each edge in the graph $G$
is processed exactly once shows that the algorithm
runs in linear time $O (|\cE|)$, which completes the proof.
\end{proof}


\section{Numerical computations}
\label{sec:appl}

We are finally ready to apply all the ideas above
to a specific example. First of all, we state
a more precise version of the theorem
given in the introduction.

\begin{theorem}
\label{thm:henon2}
There exists an open set $X \subset \bR^2$
such that $H_{a,b} (X) \subset X$
and $H_{a,b}|_X$ is (combinatorially) mixing
at all resolutions $> 10^{-5}$
for all $(a, b)$ in an open set $P \subset \bR^2$
containing $(1.4, 0.3)$.
\end{theorem}

We choose the H\'enon map for the so-called ``classical''
parameter values simply because it is a very well known
and well studied example for which nothing substantial is known.
We have purposefully formulated all our definitions
and results so far in full generality so that it
is straightforward how to apply them to many other examples.
What is left is to construct
a representation $\cF$ of the H\'enon map
such that $\cR^+ (\cF) \leq \varepsilon$
and then apply the algorithms described above to show that
$\cF$ is mixing.
We describe the construction of \( \cF \)
in detail so that minor modifications should allow
essentially the same construction to work in many other cases.

Notice that the result applies to an open set of parameter values.
This is an intrinsic feature of the method
deriving from the use of interval arithmetic.
Since both numbers $1.4$ and $0.3$ are not representable
in the binary floating-point arithmetic,
small intervals containing these numbers are taken
for the actual computations.
Therefore, the statement is automatically proved
for an open set of parameters containing these numbers.
This costs nothing from the computational point of view
and in fact corresponds to the natural fact
that numerical methods necessarily only yield results
which are stable under sufficiently small perturbations.
%

\subsection{General strategy}
\label{sec:general}

We describe in detail the strategy for a general continuous map
$f \colon \bR^n \to \bR^n$.

\subsubsection{Open interval arithmetic}
\label{sec:interval}

For the purpose of the numerical computations,
we use a slight modification of interval analysis
introduced in~\cite{Moo66}.

In the typical approach,
closed intervals are used instead of numbers,
and the result of each arithmetic operation on such intervals
is defined as the smallest possible interval
containing the results of the operation on any numbers
taken from each of the intervals,
e.g., $[a_1,a_2] - [b_1,b_2] = [a_1 - b_2, a_2 - b_1]$.
Since we are interested in doing these computations
using a fixed-size floating point representation of real numbers,
the set of actual numbers that can be represented is finite,
and thus we must round the endpoints of the resulting intervals
to the nearest representable number in the downwards direction
for the left endpoint, and in the upwards direction for the right one.
As a consequence, the result of calculations carried out
on intervals is an interval that contains every possible
exact result of those operations on the numbers
belonging to these intervals.

Since we work with open covers, in our case it is necessary
to work with open intervals, instead of closed ones.
This implies slight differences in the case of some
arithmetic operations on intervals.
Although addition, subtraction, multiplication and division
are the same, those operations in which closed intervals
would arise must be slightly changed, like rising to even powers
(e.g., computing $x^2$) or some trigonometric functions
(mainly $\sin$ and $\cos$).
Namely, the result of an operation on open intervals
with representable endpoints
is defined as the smallest \emph{open} interval
with representable endpoints which contains
the results of the operation on single elements
taken from these intervals,
e.g., $(-1,1)^2 = (-\varepsilon, 1)$,
where $-\varepsilon < 0$ is the largest representable
negative number (the set of representable numbers is finite,
so this number is well defined).

\subsubsection{Open cover parameters}
\label{sec:covparam}

We start by selecting a bounded rectangular area
$B := (a_1, a_1 + w_1) \times \cdots
\times (a_n, a_n + w_n) \subset \bR^n$
which is assumed to contain the dynamics of our interest.
We cover this area with a finite family of overlapping open boxes
($n$-dimensional hypercubes). We set some number $p_1 \in \bN$
of parts into which we intend to subdivide the first interval
$(a_1, a_1 + w_1)$ in the definition of $B$, and then we compute
the numbers $p_2, \ldots, p_n$
so that each $p_i$ is approximately proportional to $w_i$
for $i = 1, \ldots, n$.
Then we choose some small \( \kappa > 0 \)
to be the size of the overlapping margin. As a general rule
we can choose \( \kappa \) to be the smallest positive
representable number in the chosen computer precision
(recall that the set of representable numbers in \( (0,1) \)
is finite). We then consider the family
\[
\cB :\approx \big\{ \Pi_{i = 1}^{n}
\big(a_i + k_i w_i / p_i,
a_i + (k_i + 1) w_i / p_i + \kappa\big)
: k_i = 0, \ldots, p_i - 1 \big\},
\]
where the \( :\approx \) symbol is used to indicate the fact
that the actual endpoints of the intervals are computed
in the floating point arithmetic with rounding the result
of each operation to the nearest representable number
either in the downwards direction (for the left endpoints)
or in the upwards direction (for the right endpoints).

\subsubsection{Construction of $X$}
\label{sec:constrX}

The final cover \( \cU \), and thus the actual space $X$,
is chosen by an iterative procedure
from the elements of \( \cB \), so that
\( \cU \subset \cB \). In particular, this cover is
necessarily essential.

The iterative procedure is carried out as follows.
We start with a point $x_0 \in \bR^n$
as close to the attractor as possible.
This point may be obtained for example by some
non-rigorous numerical simulation of the dynamics,
but the construction is in general not particularly sensitive
to this choice.
We then define an initial approximation of the cover by letting
\[
\cU_0 := \{U \in \cB : x_0 \in U\}.
\]
For each \( U \in \cU_{0} \)
we compute an open set \( \widehat{f(U)} \)
as the image of \( U \) under \( f \)
using interval arithmetic. This is a rigorous
upper bound for \( f(U) \) in the sense that
\[
f(U) \subset \widehat{f(U)}.
\]
We then define the multivalued map
\(\cF_0 \colon \cU_{0} \multimap \cB \) by
\[
\cF_0 (U) := \{B \in \cB : B \cap \widehat{f (U)} \neq \emptyset\}.
\]
At this point we compare \( \cU_{0} \) with
\( \cF_0 (\cU_{0}) = \bigcup_{U\in\cU_{0}}\cF_0(U) \).
If \( \cF_0 (\cU_{0}) \subset \cU_{0} \) then we let
\( \cU := \cU_{0}\) and define
\( \cF := \cF_0|_{\cU} \colon \cU \multimap \cU \).
This is our combinatorial representation
of the map \( f \colon X \to X \) on the set \( X := |\cU| \).
If \( \cF_0(\cU_{0}) \not \subset \cU_{0} \) then we define
\[
\cU_{1} := \cU_{0}\cup \cF_0 (\cU_{0})
\]
and we compute the map $\cF_1 \colon \cU_1 \multimap \cB$
as an extension of $\cF_0$ using interval arithmetic.
Note that only the images of the sets in
$\cF_0 (\cU_0) \setminus \cU_0$ have to be computed.
Proceeding again in the same way, we repeat this procedure
and we obtain \(\cU_{0},\ldots, \cU_{k}\),
as well as $\cF_0, \ldots, \cF_k$, until we eventually get
\( \cF(\cU_{k}) \subset \cU_{k} \), at which point we define
\[
\cU := \cU_{k}
\quad \text{ and } \quad
\cF := \cF_k|_{\cU} \colon \cU \multimap \cU.
\]

Although this process is guaranteed to terminate,
because of the finiteness of~$\cB$, the obtained result
might be faulty if $\widehat{f (U)} \not \subset B$
for some $U \in \cU$,
as then $\cF$ is not a combinatorial representation of $f|_X$.
If this happens then we say that the construction fails,
and one has to choose different parameters for the open cover
(Section~\ref{sec:covparam}) and try again.
In particular, if the system has an attractor
with a sufficiently large basin of attraction
and strong enough convergence,
then the computations should result in a valid map $\cF$,
provided the set $B$ is taken large enough,
the number $p_1$ is high enough,
and the accuracy of computations is good enough.

Note that if this construction succeeds
then the open set $X := |\cU|$
is contained in the attraction basin of the attractor,
$f (X) \subset |\cF (X)| \subset X$,
and $\cF$ is a combinatorial representation
of $f|_X \colon X \to X$.

\subsubsection{Verification of the properties of $\cF$}
\label{sec:verif}

The outer resolution $\cR^+ (\cF)$ of the map~$\cF$
can be easily calculated using interval arithmetic
during the computation of $\cF_i$,
based on the diameter of a cover of $\widehat{f (U)}$
whenever $\cF_i (U)$ is constructed.

After the combinatorial map $\cF$ has been constructed,
one can verify whether it is transitive and mixing
by a straightforward application of Algorithms
\ref{alg:conn} and~\ref{alg:period}.

\subsection{Application to the H\'enon map}

We now give the specific data involved in the calculations
and discuss in more details the numerical
and computational issues involved.

\subsubsection{Constructing a combinatorial representation}

We closely follow the general strategy described
in Section~\ref{sec:general}.
We choose the rectangular region $B$ in such a way
that an approximation of the attractor
found in numerical simulations is contained in $B$
with some safety margin (about $0.1$ at each side).
More precisely, we set
\begin{eqnarray*}
a_1 :\approx -1.4, & \quad &
w_1 :\approx 2.8, \\
a_2 := -0.5, & \quad &
w_2 := 1.0,
\end{eqnarray*}
where the symbol ``$:\approx$'' indicates that we take
a representable number close to the given decimal number
(it is most likely the nearest representable number,
depending on a particular compiler, but it does not really matter).
Note that the numbers which are powers of $2$ are representable
and thus the actual values of $a_2$ and $w_2$
used in the calculations are exact.
We choose $\kappa$ to be the smallest
positive (normal) number representable
in the standard double precision floating point arithmetic,
which is approximately $2 \cdot 10^{-308}$.

The number $p_1$ is a parameter of the program that can be changed
at each run (it is set up from the command line
each time the program is launched),
and we figured out by trial and error
that for $p_1 := 2{,}132{,}419$
(for which the corresponding $p_2$ was taken as $761{,}578$)
the algorithm constructs a map $\cF$
with $\cR^+ (\cF) < 10^{-5}$.

As an initial point supposedly very close to the attractor
we take
\[
x_0 :\approx (0.61989426930989, 0.17586130934794).
\]
This point was found in numerical simulations
by iterating the origin a little over 100 million times.

The parameters of the H\'enon map for which we want to carry out
the computations, $a = 1.4$, $b = 0.3$, are not representable
in the binary floating point arithmetic.
Therefore, we take open intervals containing them,
computed in the program as the quotients $14/10$ and $3/10$,
respectively, with the left endpoints of the intervals rounded down
to the nearest representable numbers, and the right endpoints
rounded up.

For each $U \in \cB$, which is a product of two open intervals,
we calculate \( \widehat{f(U)} \)
by means of open interval arithmetic,
as described in Section~\ref{sec:interval},
using the formula for the H\'enon map
and the intervals containing the parameters $a$ and $b$,
as explained above.

Instead of constructing subsequent maps $\cF_i$,
the algorithm is set up in such a way
that it constructs a list $\cU$
containing the elements of the initial cover $\cU_0$
(which covers the initial point $x_0$),
and then for each element of this list
it computes $\cF (U)$ and immediately appends
the elements of $\cF (U)$ which are not yet in the list $\cU$
to the end of this list.
The algorithm terminates if the end of the list has been reached,
which means that all the elements of $\cF (U)$
for every $U$ in the list are already in the list.
The algorithm quits with a failure result
if \( \widehat{f(U)} \not \subset B \).
A sample result of the constructed set $\cU$
is illustrated in Figure~\ref{fig:hencover}.

\begin{figure}[htbp]
\includegraphics[height=7cm]{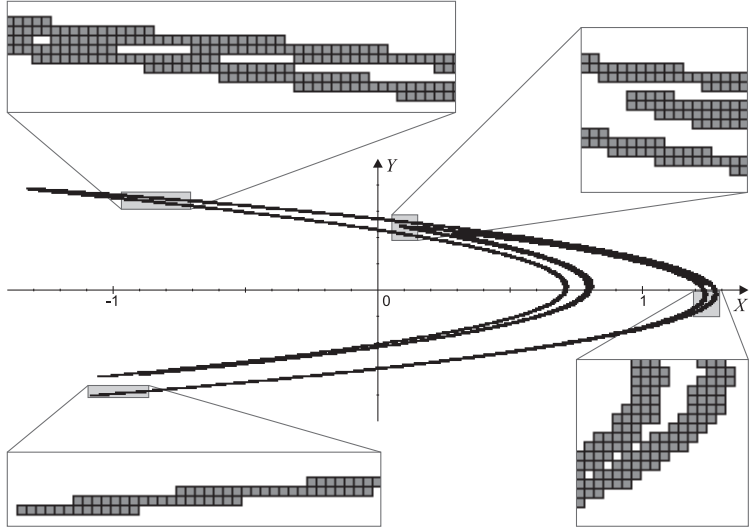}
\label{fig:hencover}
\caption{A sample low-resolution cover of the H\'enon attractor
with close-ups of selected details.}
\end{figure}

\subsubsection{Applying the graph algorithms}

The constructed cover $\cU$ and combinatorial map $\cF$
are represented in two data structures:
an array of pairs of open intervals
which represent the elements of $\cU$
(that are products of these pairs of intervals),
and a directed graph whose vertices are the indices
to this array.
After the map $\cF$ has been constructed,
the memory used for the array representing $\cU$ is released,
and the graph algorithms are run on the graph representation of $\cF$.

In the actual implementation of Algorithms \ref{alg:scc}
and~\ref{alg:period}, instead of using the recursive call
to the subroutines {\tt DFS-Visit} and {\tt tarjan},
respectively, we use a stack version of DFS because of the limitations
of some systems which allow for a limited recursion depth only.
In this version, instead of calling the subroutine,
the current parameters are put on the stack
and another round of the main loop is started with new values,
with a return from the recursive call corresponding to taking
the previously stored values of the variables from the stack.
This is a standard technique, and one can check the details
of its implementation in our software available at~\cite{WWW}.

\subsubsection{The cost of the computations}
\label{sec:cost}

As it should be expected,
the time and memory usage of the computations for the H\'enon map
heavily depend on the number $p_1$.
For small values of $p_1$, up to some $20{,}000$,
the time of computation on a computer
with a contemporary modern processor
(Intel\textregistered{} Xeon\textregistered{} 5030 2.66 GHz
was used in our computations)
should not exceed 10~seconds and use a negligible amount
of memory (up to about 25~MB).
The smallest number $p_1$ for which the computations were
successful was $p_1^{\min} = 446$
(with the corresponding $p_2^{\min} = 159$,
and the computed $\cR^{+} (\cF^{\min})$ was below $0.05$.
For smaller values of $p_1$ (and also for a few larger ones)
the algorithm fails because of encountering the situation
in which \( \widehat{f(U)} \not \subset B \)
for some $U$ in the cover being constructed.

For higher values of $p_1$,
we observed that if $p_1$ is increased $10$ times
then the outer resolution of the computed map $\cF$
decreases about $10$ times,
while the number of elements in the constructed cover $\cU$
grows about $20$ times,
and so do both the computation time and memory usage
(see the actual results available at~\cite{WWW}).
In particular, for the highest tested $p_1^{\max} := 2{,}132{,}419$
(with the corresponding $p_2^{\max} = 761{,}578$),
a cover $\cU^{\max}$ consisting of $161{,}448{,}094$ boxes
was constructed, and $\cR^{+} (\cF^{\max})$
was slightly below $10^{-5}$.
The computation time was about $1$ hour and $17$ minutes,
out of which the majority was used for the construction
of the cover $\cU$ and the graph associated to the map $\cF$,
while running the graph algorithms took less than $6$ minutes.

In these computations we reached the limits
of the computer equipment available to us
at the time of writing this paper,
but with the development of technology
one can speculate on the possibilities
of doing more extensive computations in the future.
Based on our rough estimates, obtaining the resolution
$10^{-6}$ would require some $26$~hours of processor time
and about $360$~GB of memory
(twice more than $20 \cdot 9 = 180$~GB because of the need
to switch from $32$-bit to $64$-bit integers
in the graph representation,
due to the high number of vertices and edges).
Getting down to $10^{-7}$
would take over $500$~hours ($3$~weeks)
and use some $7$~TB of memory.
Without switching to distributed computations
and probably also changing the approach
(e.g., re-computing the graph on-the-fly
instead of storing it in the memory),
these and higher resolutions seem to be still out of reach
for many years to come.
However, one might argue that such high precision
of the results is not necessary for real applications,
and $10^{-5}$ is more than one might require,
so investing in better results does not make sense at this point.

Finally, we would like to point out that the method we have developed
is dimension-independent, and so is the related software
available at~\cite{WWW}. However, the number of elements
in a typical cover of a complicated attractor should be expected
to grow considerably faster in higher dimensions
with the increase of the target resolution.
This might be a major limitation of the applicability of our method
for obtaining results at very fine resolutions
for more complicated systems.
In practice, however, even complicated dynamics
in high-dimensional spaces
often concentrates around low-dimensional attractors
(like the topological horseshoes),
in which case the resolutions interesting from the point of view
of applications might be still within the reach
of contemporary computers, and we sincerely hope to see
such applications in the future.


\section{Final remarks: Finite resolution
versus topological and measurable dynamics}

It is tempting to try to see finite resolution properties
as approximations of the ``real'', for example topological,
properties of the underlying map \( f\colon X \to X \).
In terms of the transitivity and mixing properties considered here,
this is partly justified in the sense that if \( f \)
is topologically transitive or mixing
then it is also combinatorially transitive or mixing, respectively,
at all resolutions (for completeness, we give a formal proof
in Appendix~\ref{app:mixing}).
However our combinatorial notions of transitivity and mixing
are not uniquely defined by this requirement, and to fully justify
the definitions and the idea that we are approximating
the topological properties we would need to show the converse result
that if \( f \) is combinatorially transitive or mixing
at all resolutions then it is topologically transitive or mixing.
It is possible to obtain such a double implication for certain kinds
of properties, see for instance the comprehensive discussion
on precisely this topic in \cite{Mro96}.
However, as mentioned in the introduction,
this can be achieved only for ``robust'' properties
which are persistent under small perturbations of the system.
It seems therefore that in general the finite resolution
dynamical properties of a system cannot be thought of
as an approximation of the topological properties,
at least not in a naive sense.

For a time we struggled with this limitation of the theory
and perceived it as a weakness. However, on further reflection
we realized that it was in fact arising from an incorrect
understanding of what this approach is actually about.
The mistake was to assume that the topological point of view
is in fact the ``real'' dynamics and thus the ultimate goal.
In fact, on a fundamental level, the real dynamics,
if we really want to have such a notion,
has to be just the dynamics of \( f \) thought of
as a function on the set \( X \).
Any more sophisticated description
necessarily relies on additional structure
and the description then necessarily has to be
in a form which is compatible with this structure.
There are at least two major frameworks or structures through which,
in many cases, one and the same dynamical system can be studied:
the \emph{topological} and the \emph{measurable},
which provide in some sense two alternative points of view
on the same dynamics. Each of these points of view
comes with its own definitions, notions,
and results which are intrinsically motivated
within that particular framework.
A good example, very relevant to the present paper,
is provided by the notions of topological mixing
and measure-theoretic mixing. Both of these notions
formalise some intuitive notion of ``mixing''
within the corresponding framework,
but there is no direct relationship or formal implication
between them in general (some systems may be topologically mixing
but not measure-theoretically mixing and vice versa).
Analogously, the notion of mixing which we give in the paper
is a combinatorial notion which is very natural
in the finite resolution setting.
The relationship of this notion with those of topological
and measure-theoretical mixing is interesting
and quite complex, and certainly deserving of further thought,
though beyond the scope of this paper.

In conclusion, we believe that \emph{the framework
of finite resolution dynamics provides one
of several possible structures through which to study the dynamics.
Accordingly, the definitions and dynamical features of interest
should be intrinsically motivated
within the finite resolution framework.}
Thus finite resolution dynamics
should be seen as an alternative structure,
alongside the topological and the measurable,
which can contribute to an effective study
of a dynamical system from a different point of view.


\appendix

\section{Essential covers}
\label{app:esscov}

\begin{proposition}
From any cover it is possible to create an essential one
(without increasing $\cR^+$,
but possibly decreasing the number of elements).
\end{proposition}

\begin{proof}
Let $\cU^0 = (U_1^0, \ldots, U_n^0)$ be an ordered cover
that is not necessarily essential. We shall construct a series
of gradually ``corrected'' ordered covers
$\cU^1, \ldots, \cU^n$ such that each of the covers
has exactly $n$ elements, $\cU^k = (\cU_1^k, \ldots, \cU_n^k)$,
and each of the first $k$ elements of $\cU^k$
either is empty, or satisfies the exclusive ball condition
that appears in Definition~\ref{def:esscov} of an essential cover.
Then $\cU^n$ will give rise to an essential cover
after having removed the empty sets.
We proceed by induction.
Note that $\cU^0$ satisfies the inductive assumption.
Suppose that $\cU^k$ has been constructed for some $k < n$.
Then each $U_i^k$ for $i \leq k$ is either empty
or satisfies the exclusive ball condition,
and then we set $U_i^{k + 1} := U_i^k$ for $i \leq k$.
Consider $U_{k + 1}^k$. If it is empty or satisfies
the exclusive ball condition then take $U_i^{k + 1} := U_i^k$
also for all $i = k + 1, \ldots, n$.
If $U_{k + 1}^k \subset \bigcup_{i \neq k + 1} U_i^k$
then set $U_{k + 1}^{k + 1} := \emptyset$
and $U_i^{k + 1} := U_i^k$ for all $i = k + 2, \ldots, n$.
Otherwise, take any $x \in U_{k + 1}^k$.
Since $U_{k + 1}^k$ is open, there exists $r > 0$
such that $B (x, r) \subset U_{k + 1}^k$.
Take $U_{k + 1}^{k + 1} := U_{k + 1}^k$
and $U_i^{k + 1} := U_i^k \setminus \cl B (x, r / 2)$
for $i = k + 2, \ldots, n$.
Note that all $U_i^{k + 1}$, $i = 1, \ldots, n$, are open sets
(although some might be empty),
and $\bigcup_{i = 1}^n U_i^k = \bigcup_{i = 1}^n U_i^{k + 1}$,
so $\cU^{k + 1}$ is an ordered cover of the same space.
Since $U_i^{k + 1} = U_i^k$ for $i \leq k$
and $U_i^{k + 1} \subset U_i^k$ for $i > k$,
the exclusive ball condition for each nonempty $U_i^{k + 1}$
with $i \leq k$ follows directly from the one
for the corresponding $U_i^k$.
Moreover, by the construction,
$U_{k + 1}^{k + 1}$ is either empty,
or satisfies the exclusive ball condition with $B (x, r / 2)$.
\end{proof}

\section{Topological mixing implies finite resolution mixing}
\label{app:mixing}

\begin{proposition}
Suppose \( X \) is a metric space and \( f\colon X \to X \)
a continuous map which is topologically transitive
or topologically mixing.
Then any combinatorial representation \( \mathcal F \) of \( f \)
is transitive or mixing, respectively.
\end{proposition}

\begin{proof}
Recall that \( f \) $f$ is \emph{transitive} if for every two
open sets $U, V \subset X$ there exists $n > 0$
such that $f^{-n} (U) \cap V \neq \emptyset$.
Let $\cF \colon \cU \multimap \cU$
be a combinatorial representation
of a transitive map $f \colon X \to X$.
We shall show that $\cF$ is transitive.
Take any $U, V \in \cU$.
Since $f$ is transitive, there exists $n > 0$
such that $f^{-n} (U) \cap V \neq \emptyset$.
Therefore, $f^n (f^{-n} (U)) \cap f^n (V) \neq \emptyset$.
Since $f^n (f^{-n} (U)) \subset U$,
we have $U \cap f^n (V) \neq \emptyset$.
Note that $\cF^n$ is a representation of $f^n$,
and thus $U \in \cF^n (V)$.
Since the choice of $U$ and $V$ was arbitrary,
this implies that $\cF$ is transitive.

Similarly, $f$ is \emph{mixing} if for every two
open sets $U, V \subset X$ there exists $N > 0$
such that $f^{-n} (U) \cap V \neq \emptyset$
for all $n > N$.
Let $\cF \colon \cU \multimap \cU$ be a combinatorial
representation of a mixing map $f \colon X \to X$.
We shall show that $\cF$ is mixing.
Let $U, V \in \cU$.
Since $f$ is mixing, there exists $N > 0$
such that $f^{-n} (U) \cap V \neq \emptyset$
for all $n > N$. Fix this $n$.
Note that $f^n (f^{-n} (U)) \cap f^n (V) \neq \emptyset$.
Since $f^n (f^{-n} (U)) \subset U$,
we have $U \cap f^n (V) \neq \emptyset$.
Recall that $\cF^n$ is a representation of $f^n$,
and thus $U \in \cF^n (V)$.
Since the choice of $U$ and $V$ was arbitrary,
this proves that $\cF$ is mixing.
\end{proof}



\begin{thebibliography}{10}

\bibitem{Ara07}
{\sc Z. Arai},
\newblock On hyperbolic plateaus of the H{\'e}non map,
\newblock {\em Experiment. Math.}, 16 (2007), 181--188.

\bibitem{AraKalKokMisOkaPil09}
{\sc Z. Arai, W. Kalies, H. Kokubu, K. Mischaikow,
 H. Oka and P. Pilarczyk},
\newblock A database schema for the analysis of global dynamics
 of multiparameter systems,
\newblock {\em SIAM J. Appl. Dyn. Syst.} 8 (2009), 757--789.

\bibitem{AraMis06}
{\sc Z. Arai and K. Mischaikow},
\newblock Rigorous computations of homoclinic tangencies.
\newblock {\em SIAM Journal on Applied Dynamical Systems},
 5 (2006) 280--292.

\bibitem{ArbMat04}
{\sc A. Arbieto and C. Matheus},
\newblock Decidability of Chaos for some families of dynamical systems,
\newblock {\em Found. Comput. Math. } (2004), 269--275.

\bibitem{BenCar91}
{\sc M. Benedicks and L. Carleson},
\newblock The dynamics of the {H}\'enon map,
\newblock {\em Ann. of Math. (2)} 133 (1991), 73--169.

\bibitem{BenYou93}
{\sc M. Benedicks and L.S. Young},
\newblock Sinai-Bowen-Ruelle measures for certain H{\'e}non maps,
\newblock {\em Invent. Math.}, 112 (1993), 541--576.

\bibitem{Bow70}
{\sc R. Bowen},
\newblock Markov partitions for {A}xiom {${\rm A}$} diffeomorphisms,
\newblock {\em Amer. J. Math.} 92 (1970), 725--747.

\bibitem{Bow75}
{\sc R. Bowen},
\newblock {\em Equilibrium states and the ergodic theory
 of {A}nosov diffeomorphisms},
\newblock Lecture Notes in Mathematics,
 Vol. 470. Springer-Verlag, Berlin, 1975.

\bibitem{BowRue75}
{\sc R. Bowen and D. Ruelle},
\newblock The ergodic theory of {A}xiom {A} flows,
\newblock {\em Invent. Math.}, 29 (1975), 181--202.

\bibitem{Col98}
{\sc E. Colli},
\newblock Infinitely many coexisting strange attractors,
\newblock {\em Ann. Inst. H. Poincar{\'e} Anal. Non Lin{\'e}aire},
 15(5) (1998), 539--579.

\bibitem{CorLeiRivSte01}
{\sc T.~H. Cormen, C.~E. Leiserson, R.~L. Rivest and C. Stein},
\newblock {\em Introduction to algorithms},
\newblock MIT Press, Cambridge, MA, second edition, 2001.

\bibitem{EckKocWit84}
{\sc J.-P. Eckmann, H. Koch and P. Wittwer},
\newblock A computer-assisted proof of universality
 for area-preserving maps,
\newblock {\em Mem. Amer. Math. Soc.}, 47(289):vi+122, 1984.

\bibitem{Gal02}
{\sc Z. Galias},
\newblock Rigorous investigation of the Ikeda map by means of interval
 arithmetic,
\newblock {\em Nonlinearity}, 15 (2002), 1759--1779.

\bibitem{GonTurShi03}
{\sc S.~V. Gonchenko, D.~V. Turaev and L.~P. Shilprimenikov},
\newblock On the dynamic properties of diffeomorphisms with homoclinic
 tangencies,
\newblock {\em Sovrem. Mat. Prilozh.} 7 (2003), 91--117.

\bibitem{Hen76}
{\sc M. H{\'e}non},
\newblock A two-dimensional mapping with a strange attractor,
\newblock {\em Comm. Math. Phys.}, 50 (1976), 69--77.

\bibitem{Hru06}
{\sc S. Lynch Hruska},
\newblock A numerical method for constructing the hyperbolic structure of
 complex H{\'e}non mappings.
\newblock {\em Found. Comput. Math.}, 6 (2006) 427--455.

\bibitem{JabKul03}
{\sc D. Jab\l{}o\'nski and M. Kulczycki},
\newblock Topological transitivity, mixing and nonwandering set
 of subshifts of finite type -- a numerical approach,
\newblock {\em Intern. J. Computer Math.} 80 (2003), 671--677.

\bibitem{Jak81}
{\sc M.~V. Jakobson},
\newblock Absolutely continuous invariant measures
 for one-parameter families of one-dimensional maps,
\newblock {\em Comm. Math. Phys.}, 81(1) (1981), 39--88.

\bibitem{KalMisVan05}
{\sc W.~D. Kalies, K. Mischaikow and R.~C.~A.~M. VanderVorst},
\newblock An algorithmic approach to chain recurrence,
\newblock {\em Found. Comput. Math.} 5 (2005), 409--449.

\bibitem{KapSim07}
{\sc T. Kapela and C. Sim{\'o}},
\newblock Computer assisted proofs for nonsymmetric
 planar choreographies and for stability of the eight,
\newblock {\em Nonlinearity}, 20 (2007), 1241--1255.

\bibitem{KapZgl03}
{\sc T. Kapela and P. Zgliczy{\'n}ski},
\newblock The existence of simple choreographies for the N-body
 problem--a computer-assisted proof,
\newblock {\em Nonlinearity}, 16 (2003), 1899--1918.

\bibitem{Lan82}
{\sc O. Lanford},
\newblock A computer-assisted proof of the Feigenbaum conjectures,
\newblock {\em Bull. Amer. Math. Soc. (N.S.)} 6 (1982), 427--434.

\bibitem{Lor63}
{\sc E.~N. Lorenz},
\newblock Deterministic non-periodic flows,
\newblock {\em J. Atmos. Sci.} 20 (1963), 130--141.

\bibitem{LuzMelPac05}
{\sc S. Luzzatto, I. Melbourne and F. Paccaut},
\newblock The {L}orenz attractor is mixing,
\newblock {\em Comm. Math. Phys.} 260 (2005), 393--401.

\bibitem{LuzTak06}
{\sc S. Luzzatto and H. Takahasi},
\newblock Computable conditions for the occurrence
 of non-uniform hyperbolicity in families of one-dimensional maps,
\newblock {\em Nonlinearity} 19 (2006), 1657--1695.

\bibitem{LuzTuc99}
{\sc S. Luzzatto and W. Tucker},
\newblock Non-uniformly expanding dynamics
 in maps with singularities and criticalities,
\newblock {\em Inst. Hautes Etudes Sci. Publ. Math.} 89 (1999), 179--226.

\bibitem{LuzVia00}
{\sc S. Luzzatto and M. Viana},
\newblock Positive Lyapunov exponents for Lorenz-like families with
 criticalities,
\newblock {\em Ast{\'e}risque}, (261):xiii, 201--237, 2000.

\bibitem{MisMro95}
{\sc K. Mischaikow and M. Mrozek},
\newblock Chaos in Lorenz equations: A computer assisted proof,
\newblock {\em Bull. Amer. Math. Soc. (N.S.)} 33 (1995), 66--72.

\bibitem{Moo66}
{\sc R.~E. Moore}, {\em Interval analysis}, Prentice-Hall, Inc.,
 Englewood Cliffs, N.J., 1966.

\bibitem{Mro96}
{\sc M. Mrozek},
\newblock Topological invariants, multivalued maps
 and computer assisted proofs in dynamics,
\newblock {\em Computers Math. Applic} 32 No. 4 (1996), 83-104.

\bibitem{Mro99}
{\sc M. Mrozek},
\newblock An algorithm approach to the Conley index theory,
\newblock {\em J. Dynam. Differential Equations} 11 (1999), 711--734.

\bibitem{MroPil02}
{\sc M. Mrozek and P. Pilarczyk},
\newblock The Conley index and rigorous numerics
 for attracting periodic orbits,
\newblock {\em Proceedings of the Conference
on Variational and Topological Methods in the Study
of Nonlinear Phenomena (Pisa, 2000)}, 65--74,
Progr. Nonlinear Differential Equations Appl., 49,
Birkh\"auser Boston, Boston, MA, 2002.

\bibitem{MorVia93}
{\sc L. Mora and M. Viana},
\newblock Abundance of strange attractors,
\newblock {\em Acta Math.} 171 (1993), 1--71.

\bibitem{New74}
{\sc S.E. Newhouse},
\newblock Diffeomorphisms with infinitely many sinks,
\newblock {\em Topology}, 13 (1974), 9--18.

\bibitem{PacRovVia98}
{\sc M.J. Pacifico, A. Rovella and M. Viana},
\newblock Infinite-modal maps with global chaotic behavior,
\newblock {\em Ann. of Math. (2)}, 148 (1998), 441--484.

\bibitem{PalTak93}
{\sc J. Palis and F.Takens},
\newblock Hyperbolicity and sensitive chaotic dynamics at homoclinic
 bifurcations, Cambridge University Press, Cambridge, 1993.

\bibitem{Pil99}
{\sc P. Pilarczyk},
\newblock Computer assisted method for proving existence
 of periodic orbits,
\newblock {\em TMNA} 13 (1999), 365--377.

\bibitem{WWW}
{\sc P. Pilarczyk},
\newblock Finite resolution dynamics. Software and examples,
\newblock {\tt http://www.pawelpilarczyk.com/finresdyn/}.

\bibitem{Pil03}
{\sc P. Pilarczyk},
\newblock Topological-numerical approach to the existence
 of periodic trajectories in ODEs,
\newblock {\em Discrete and Continuous Dynamical Systems 2003},
 A Supplement Volume: Dynamical Systems and Differential Equations,
 701--708.

\bibitem{PilSto08}
{\sc P. Pilarczyk and K. Stolot},
\newblock Excision-preserving cubical approach to the algorithmic
 computation of the discrete Conley index,
\newblock {\em Topology and its Applications} 155 (2008), 1149--1162.

\bibitem{Sin68}
{\sc J.~G. Sina{\u\i}},
\newblock Markov partitions and {U}-diffeomorphisms,
\newblock {\em Funkcional. Anal. i Prilo\v{z}en} 2 (1968), 64--89.

\bibitem{Szy97}
{\sc A. Szymczak},
\newblock A combinatorial procedure for finding
 isolating neighbourhoods and index pairs,
\newblock {\em Proc. Royal Soc. Edinburgh Sect. A} 127 (1997), 1075--1088.

\bibitem{Tar72}
{\sc R. Tarjan},
\newblock Depth-first search and linear graph algorithms,
\newblock {\em SIAM J. Comput.} 1 (1972), 146--160.

\bibitem{Thu99}
{\sc H. Thunberg},
\newblock Positive exponent in families with flat critical point,
\newblock {\em Ergodic Theory Dynam. Systems}, 19 (1999) 767--807.

\bibitem{Tsu93}
{\sc M. Tsujii},
\newblock Positive Lyapunov exponents in families
 of one-dimensional dynamical systems,
\newblock {\em Invent. Math.}, 111 91993) 113--137.

\bibitem{Tuc99}
{\sc W. Tucker},
\newblock The Lorenz attractor exists,
\newblock {\em C.R. Acad. Sci. Paris}, S\'erie~I
 328 (1999), 1197--1202.

\bibitem{WanYou01}
{\sc Q. Wang and L.-S. Young},
\newblock Strange attractors with one direction of instability,
\newblock {\em Comm. Math. Phys.} 218 (2001), 1--97.

\bibitem{AperWiki}
{\sc Wikipedia contributors},
\newblock Aperiodic graph,
\newblock {\em Wikipedia{,} The Free Encyclopedia} (2008)
{\tt http://en.wikipedia.org/w/index.php?title=%
 Aperiodic\_graph\&oldid=224030707}.

\bibitem{TarWiki}
{\sc Wikipedia contributors},
\newblock Tarjan's strongly connected components algorithm,
\newblock {\em Wikipedia{,} The Free Encyclopedia} (2009)
{\tt http://en.wikipedia.org/w/index.php?
title=Tarjan\%27s\_strongly\_connected\_components\_%
algorithm\&oldid=295022950}.

\bibitem{Zgl04}
{\sc P. Zgliczy{\'n}ski},
\newblock Rigorous numerics for dissipative
 partial differential equations.
 II. Periodic orbit for the Kuramoto-Sivashinsky
 PDE--a computer-assisted proof,
\newblock {\em Found. Comput. Math.}, 4 (2004), 157--185.

\end{thebibliography}

\end{document}